\documentclass[11pt]{article}
\usepackage[centering,top=1.5in,bottom=1.2in,left=1.1in,right=1.1in]{geometry}
\usepackage{palatino}
\usepackage[utf8]{inputenc}
\usepackage{verbatim}
\usepackage{graphicx,subcaption}
\usepackage{graphbox}
\usepackage{amsmath,amssymb,amsthm}
\usepackage{hyperref}
\usepackage{bm}
\usepackage{enumitem}
\usepackage[usenames]{color}
\usepackage[dvipsnames]{xcolor}
\usepackage{listings}
\usepackage{algorithm,algorithmic}
\usepackage{caption} 
\usepackage{extarrows}
\usepackage{newtxmath}
\usepackage{tikz}
\usetikzlibrary{shapes.geometric, arrows}

\newenvironment{@abssec}[1]{%
     \if@twocolumn
       \section*{#1}%
     \else
       \vspace{.05in}\footnotesize
       \parindent .2in
         {\upshape\bfseries #1. }\ignorespaces
     \fi}
     {\if@twocolumn\else\par\vspace{.1in}\fi}
\newenvironment{abstractinline}{\begin{@abssec}{Abstract}}{\end{@abssec}}
\newenvironment{keywords}{\begin{@abssec}{Keywords}}{\end{@abssec}}
\newenvironment{ams}{\begin{@abssec}{AMS subject classifications}}{\end{@abssec}}

\newtheoremstyle{myThm} 
    {\topsep}                    
    {\topsep}                    
    {\itshape}                   
    {}                           
    {\bfseries} 
    {.}                          
    {.5em}                       
    {}  

\newtheoremstyle{myRem} 
    {\topsep}                    
    {\topsep}                    
    {}                   
    {}                           
    {\bfseries} 
    {.}                          
    {.5em}                       
    {}  

\newtheoremstyle{myDef} 
    {\topsep}                    
    {\topsep}                    
    {}                   
    {}                           
    {\bfseries} 
    {.}                          
    {.5em}                       
    {}  

\theoremstyle{myThm}
\newtheorem{theorem}{Theorem}[section]
\newtheorem{lemma}[theorem]{Lemma}

\theoremstyle{myRem}
\newtheorem{remark}[theorem]{Remark}

\theoremstyle{myDef}
\newtheorem{definition}[theorem]{Definition}
\newtheorem{example}[theorem]{Example}

\newtheorem{manuallemmainner}{Lemma}
\newenvironment{manuallemma}[1]{%
  \renewcommand\themanuallemmainner{#1}%
  \manuallemmainner
}{\endmanuallemmainner}

\newcommand{\rd}{\mathrm{d}}
\def\R{\mathbb{R}}
\def\C{\mathbb{C}}
\def\F{\mathbb{F}}

\def\M{\mathcal{M}}
\def\N{\mathcal{N}}
\def\St{\text{St}}

\def\S{\mathbb{S}}
\def\grad{\nabla}

\DeclareMathOperator*{\argmin}{arg\,min}
\def\rank{\text{rank}}
\def\U{\mathcal{U}}
\def\grad{\text{grad}}

\newcommand{\zzy}[1]{{#1}}

\newcommand{\blue}[1]{{{#1}}}

\graphicspath{{./figures/}}

\begin{document}
\title{Asymptotic Escape of Spurious Critical Points  \\
on the Low-rank Matrix Manifold}
\date{\today}
\author{Thomas Y. Hou\thanks{ACM, Caltech. Correspondence to: Ziyun Zhang (\texttt{zyzhang@caltech.edu}).}, \,
Zhenzhen Li\footnotemark[1], \,
Ziyun Zhang\footnotemark[1]}
\maketitle

\begin{abstractinline}
\blue{We show that on the manifold of fixed-rank and symmetric positive semi-definite matrices, the Riemannian gradient descent algorithm almost surely escapes some spurious critical points on the boundary of the manifold. Our result is the first to partially overcome the incompleteness of the low-rank matrix manifold without changing the vanilla Riemannian gradient descent algorithm. The spurious critical points are some rank-deficient matrices that capture only part of the eigen components of the ground truth. Unlike classical strict saddle points, they exhibit very singular behavior. We show that using the dynamical low-rank approximation and a rescaled gradient flow, some of the spurious critical points can be converted to classical strict saddle points in the parameterized domain, which leads to the desired result. Numerical experiments are provided to support our theoretical findings. 
}


\end{abstractinline} 


\begin{ams}
58D17, 65F10, 90C26, 15A23
\end{ams}

\section{Introduction}


Low-rank matrix recovery problems are prevalent in modern data science, artificial intelligence and related technological fields. The low-rank property of matrices is widely exploited to extract the hidden low-complexity structure in massive datasets from machine learning, signal processing, imaging science, advanced statistics, information theory and quantum mechanics, just to name a few. 

The {low-rank matrix manifold} 
\cite{helmke2012optimization, helmke1995critical}
has gained popularity in recent years since it gives a neat description of low-rank matrices. The set of matrices with the same size $m$ by $n$ and a fixed rank $r$ forms a smooth manifold $\M_r$, which is a nonconvex set that is locally isomorphic to the Euclidean space. Many nonconvex optimization techniques can be transferred to $\M_r$ without much difficulty. Among them, the Riemannian gradient descent, the manifold version of the vanilla gradient descent, demonstrates nearly optimal convergence rate and practical flexibility in a number of problems, \blue{see e.g. \cite{cambier2016robust, chi2019nonconvex, hou2020fast,  schneider2015convergence, vandereycken2013low, wei2016guarantees}.}

A fundamental problem has yet remained open in the global analysis of optimization on the low-rank matrix manifold. This comes from the fact that $\M_r$ is \emph{not} a complete set. The boundary of $\M_r$ consists of matrices with rank smaller than $r$, which are not in $\M_r$ themselves. In other words, $\overline{\M_r}\backslash\M_r = \cup_{s=0}^{r-1}\M_s \not\subset \M_r$. There is no guarantee that the limit point of an iterative sequence will converge to a rank-$r$ ground truth instead of being stuck at some lower-rank \emph{spurious critical points}. 

In our previous work \cite{hou2020fast}, it has been proved that under certain assumptions, the converging set of the spurious points has very small measure. This means that starting from a randomly sampled initialization, the iterative sequence avoids these spurious critical points \emph{with high probability}. However, practical applications seem to imply an even stronger result. In fact, we observe that from random initializations, Riemannian gradient descent \emph{almost surely} avoids these spurious critical points. This motivates us to conjecture that their converging sets actually have zero measure.

To understand this phenomenon, it helps to compare it with the asymptotic escape of strict saddle points by gradient descent \cite{lee2016gradient}. The two are remarkably similar, except that the spurious critical points in our context are \emph{not} strict saddles. Instead, the spurious critical points are singular points with negative infinity Hessian directions. More advanced techniques are needed to deal with their singularity.

In this paper, we give a partially confirmatory answer to the aforementioned conjecture. We show that the Riemannian gradient flow and the Riemannian gradient descent with varying stepsize asymptotically escape the rank-$(r-1)$ spurious critical points on the rank-$r$ \blue{symmetric positive semi-definite (SPSD)} manifold. We propose to use the dynamical low-rank approximation \cite{koch2007dynamical} to describe the gradient flow on the low-rank matrix manifold. We then introduce a rescaled gradient flow to remove the singularity of the ODE system. After rescaling, classical saddle escape theorems can be applied to derive the desired result.

Below is an example which illustrates that the spurious critical points can be the limit points of the Riemannian gradient descent algorithm, but the required initialization is so special that it is almost impossible under random initialization.
\begin{example}
\label{ex: spurious}
    Assume that $n=3$, $r=2$. We use the vanilla Riemannian gradient descent (Riemannian GD) algorithm $Z_{k+1} = R(Z_k - \alpha \cdot P_{T_{Z_k}}(Z_k-X))$ to minimize the least squares loss function $f(Z) = \frac{1}{2}\|Z-X\|_F^2$ on the manifold $\M_2 = \{Z: \, Z\in\R^{3\times 3}, \, \rank(Z)=2\}$. Here $P_{T_{Z_k}}(\cdot)$ is the projection onto the tangent space of $\M_2$ at $Z_k$ and $R(\cdot)$ is the retraction, cf. Section \ref{sec:manifold}. Let 
    \begin{align*}\small
        X=\begin{pmatrix}2&0&0\\0&1&0\\0&0&0\end{pmatrix}, \qquad Z_0=\begin{pmatrix}2&0&0\\0&0&0\\0&0&1\end{pmatrix}.
    \end{align*}
    Let the step size $\alpha \in (0,1)$. Then the sequence $\{Z_k\}_{k=0}^\infty$ generated by the Riemannian GD and its limit point are given by
    \begin{align*}\small
        Z_k=\begin{pmatrix}2&0&0\\0&0&0\\0&0&(1-\alpha)^k\end{pmatrix}, \qquad Z_\#:=\lim_{k\to\infty}Z_k=\begin{pmatrix}2&0&0\\0&0&0\\0&0&0\end{pmatrix}.
    \end{align*}
    We see that $Z_\#$ is a spurious critical point. Note that even though $Z_k \in \M_2$ for any $k$, $Z_\# \not\in \M_2$. Instead, $Z_\# \in \M_1$. 
    
    However, from a slightly perturbed initial point
    \begin{align*}
        \small
        Z_0=
        \begin{pmatrix}
            2 & 0 & 0\\
            0 & \epsilon^2 & \epsilon\\
            0 & \epsilon & 1
        \end{pmatrix},
    \end{align*}
    with arbitrarily small $\epsilon>0$, one can always show that $\lim_{k \to\infty} Z_k = X$, i.e., the limit point is not a spurious point.
    
    Figure \ref{fig: Zstararea} is a visualization of the gradient $\|P_{T_{Z}}(Z-X)\|_F$ in the neighborhood of a spurious $Z_\#$. We can see that the gradient is singular near $Z_\#$. There is only one direction in which the sequence converges to $Z_\#$. Along other directions, the Riemannian gradient remains large and the sequence does not converge to $Z_\#$. 
    \begin{figure}[ht]
        \centering
        \includegraphics[width=0.35\textwidth]{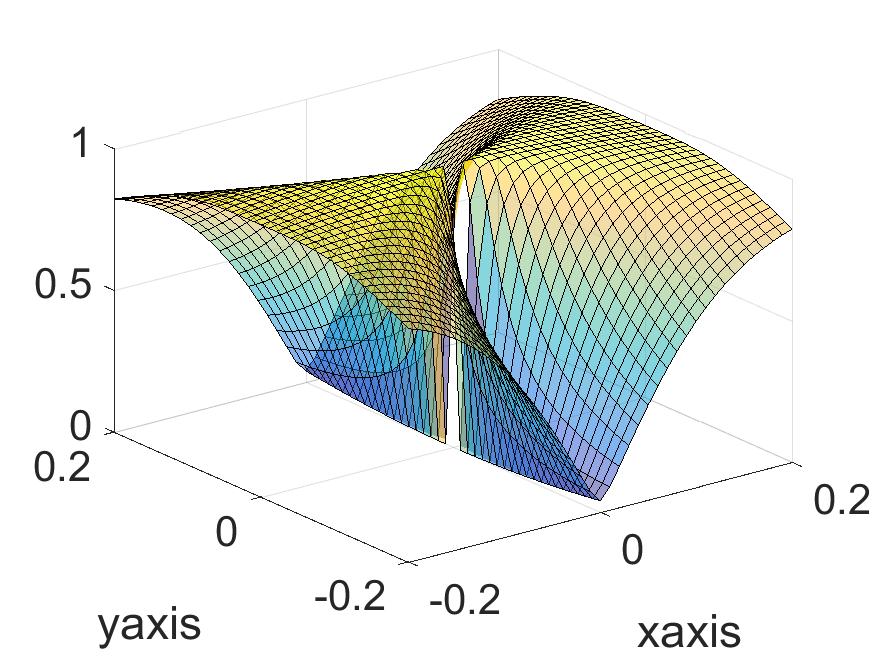}
        \caption{\small Magnitude of the gradient in the neighborhood of a spurious critical point}\label{fig: Zstararea}
    \end{figure}
\end{example}



\subsection{Related work}

\noindent\textbf{Incompleteness of the low-rank matrix manifold.} The fact that $\M_r$ is not a complete set is first reported in \cite{vandereycken2013low} in the context of matrix completion, and later in \cite{kressner2014low} with low Tucker-rank tensor completion. To guarantee that the iterative sequence of the proposed algorithm stays inside a compact subset of $\M_r$, the author of \cite{vandereycken2013low} proposes to add a regularization term to the objective function $f(Z)$:
\begin{align*}
        g(Z) = f(Z) + \mu^2 (\|Z\|_F^2 + \|Z^\dagger\|_F^2),
\end{align*}
where $Z^\dagger$ is the pseudo-inverse of $Z$, and $\mu$ is a parameter. In particular, the term $\mu^2 \|Z^\dagger\|_F^2$ guarantees that $\|Z^\dagger\|_F$ will not go to infinity, i.e. $Z$ will not go to rank lower than $r$. 
        
However, the author also comments that $\mu^2$ can be chosen very small, in fact as small as $10^{-16}$. In numerical experiments, one can simply neglect this term and use the original function $f(Z)$ instead of the regularized function $g(Z)$. In other words, the author observes that the iterative sequence of the vanilla Riemannian gradient descent almost surely avoids the rank-deficient points and stays inside $\M_r$.

\medskip
\zzy{
\noindent\textbf{Apocalypses from a geometric point of view.} 
Concurrent with our paper, the authors of  \cite{levin2021finding} propose a similar concept. They use the term \emph{apocalypse} to describe the event where the sequence of iterative points is in $\M_r$ but the limit point has rank less than $r$. This is exactly what happens in Example \ref{ex: spurious}. They observe that apocalypse occurs when the tangent cone at the limit is not contained in the limit of the tangent cones. A more detailed discussion on the relation between tangent cones and optimality conditions can be found in \cite{li2019low}.

Along this line of research, two remedies have been proposed to fix the apocalypse. The first is a second-order algorithm \cite{levin2021finding}, which uses a smooth lift (essentially the Burer-Monteiro factorization $X = UV^\top$) and the trust-region method. Another is a first-order algorithm proposed in \cite{olikier2022apocalypse}, which uses the numerical rank to perform suitable rank reductions.

We remark that although both approaches could avoid spurious points, they require major modification to the gradient descent algorithm. In contrast, we focus on explaining why gradient descent needs no modification in practice.  
}

\medskip
\noindent\textbf{Asymptotic escape of classical strict saddle points.} Gradient descent with random initialization almost surely escapes strict saddle points and converges to minimizers. Such phenomenon has been well studied in the literature. The seminal works in \cite{lee2016gradient} and \cite{lee2019first} deal with isolated strict saddles in the Euclidean space. Later the result is extended to non-isolated saddles in \cite{panageas2016gradient}, to the Riemannian manifold in \cite{criscitiello2019efficiently} and \cite{sun2019escaping}, and to the strict critical submanifold in \cite{hou2020analysis}. We cite it in its most general form in Theorem \ref{thm:escape3}. In addition to the results with fixed step size, later work also extends the result to a diminishing step size \cite{panageas2019first}.

\blue{
Our central observation is that the spurious critical points in $\overline{\M_r}\backslash\M_r$ are fundamentally different from, but subtly related to, the classical strict saddle points. The spurious critical points have singular local neighborhoods as illustrated in Figure \ref{fig: Zstararea}. Their asymptotic escape behavior cannot be directly explained by Theorem \ref{thm:escape3}. However, using a rescaled gradient flow, we can eliminate the singularity, and apply the saddle escape results to the rescaled system. 
}

\medskip
\noindent\textbf{Implicit regularization in low-rank matrix factorization.} The concept of \emph{implicit regularization} is often used to describe the emergence of favorable structures without explicit regularization terms. In deep matrix factorization and deep neural networks, this describes a tendency towards low-rank solutions and better generalization \cite{arora2019implicit}. In statistical estimation, this could mean a tendency to promote incoherence and accelerate convergence \cite{chi2019nonconvex}\cite{ma2018implicit}. As we have seen, the phenomenon that iterative sequences on the incomplete manifold $\M_r$ stay inside the manifold does not rely on an explicit regularization term $\mu^2\|Z^\dagger\|_F^2$. Thus it can also be seen as a form of implicit regularization.

\medskip
\noindent\textbf{Matrix decomposition and its continuity.} \blue{Our analysis crucially relies on finding a low-rank decomposition that is sufficiently continuous along the whole gradient flow trajectory. The dynamical low-rank approximation (DLRA), first proposed in \cite{koch2007dynamical}, is a decomposition that suits our purpose. In contrast, the singular value decomposition will lose its differentiability whenever singular values coalesce \cite{dieci1999smooth}. A variant called the \emph{analytic SVD} \cite{bunse1991numerical} could fix this issue, but it requires analyticiy of the gradient function, which cannot be satisfied by Riemannian gradients on $\M_r$. We remark that the success of DLRA is still limited to the rank-$(r-1)$ spurious critical points. Extension of the current analysis to general spurious critical points is left for future work.
}


\subsection{Organization of this paper}
The rest of this paper is organized as follows. In Section \ref{sec:prelimiaries} we introduce some preliminary results to set the stage. 
In Section \ref{sec:main} we present and prove the main result of this paper, which is the asymptotic escape of the rank-$(r-1)$ spurious critical points by the gradient flow. Specifically, we introduce the rescaled gradient flow, prove its $C^0$- and $C^1$-extension to the rank-$(r-1)$ spurious critical points, and show that these points are strict saddles under the rescaled flow. In Section \ref{sec:GD} we present the corresponding result for the gradient descent. In Section \ref{sec:numerical}, some numerical experiments are performed to illustrate our theoretical results. Finally, Section \ref{sec:discussion} is devoted to some discussions. 

\section{Preliminaries}
\label{sec:prelimiaries}

\noindent\textbf{Notations.} Unless otherwise specified, upper case letters stand for matrices, lower case letters stand for vectors or scalars, and calligraphic letters stand for manifolds or sets. The field $\F$ can be either $\R$ or $\C$. The low-rank matrix manifold is denoted as $\M_r$. The Hermitian transpose is denoted as $(\cdot)^*$. The set of $n \times n$ Hermitian matrices is denoted as $\mathbb{S}_n$. The Stiefel manifold is $\St(n,r) = \{U \in \F^{n\times r}: \, U^* U = I_r\}$. The orthogonal group is $\text{SO}(n) = \St(n,n)$. The subscript $(\cdot)_{\#}$ is reserved for the spurious critical points. We use $\text{grad}$ and $\text{Hess}$ to denote the Riemannian gradient and Hessian, and $\nabla$ to denote the Euclidean derivative.

\subsection{Manifold setting}
\label{sec:manifold}
Let $\M_r$ ($r\in \mathbb{N}$) denote the fixed rank manifold $\{Z\in\F^{m\times n}: \text{rank}(Z)=r\}$, where $\F = \R$ or $\C$.  Let $\overline{\M_r}$ be its closure. We summarize some basic properties of $\M_r$ below. \blue{A more detailed introduction can be found in \cite{hou2020analysis}.}

\begin{lemma} \leavevmode \label{lemma:manifold_basics}
Let $\M_r=\{Z\in\F^{m\times n}: \rank(Z)=r\}$, where $\F = \R$ or $\C$. Then $\overline{\M_r}=\{Z\in\F^{m\times n}: \rank(Z)\leq r \}$. Furthermore, we have the following:
    \begin{enumerate}[label={(\arabic*)}]
        \item $\M_r$ is dense in $\overline{\M_r}$.
        \item For general $\M_r\subset\F^{m\times n}$ of non-Hermitian matrices, $\M_r$ is connected. If restricted to $m=n$, $\M_r \subset \S_{n}$ Hermitian, then $\M_r$ has $r+1$ disjoint branches and each branch is connected.
        \item The local dimension of $\M_r$ is \begin{align*}
            \text{dim}(\M_r) = 
            \begin{cases}
                (m+n-r)r, & \F = \R, \text{ non-Hermitian;} \\
                (2m+2n-r)r, & \F = \C, \text{ non-Hermitian;}  \\
                \frac{(2m-r+1)r}{2}, &  \F = \R, \text{ Hermitian;}\\
                \frac{(4m-r+1)r}{2}, &  \F = \C, \text{ Hermitian.}\\
            \end{cases}
        \end{align*}
        \item The boundary of $\M_r$ is $\overline{\M_r}\setminus\M_r=\cup_{s=0}^{r-1}\M_s$.
    \end{enumerate}
\end{lemma}

\begin{lemma}[Tangent space of $\M_r$]
    Let $X\in\M_r$, $X=U\Sigma V^*$. Let $\mathcal{U}=\text{Col}(U)$, $\mathcal{V}=\text{Col}(V)$ be the column spaces of $U$ and $V$ respectively. Then the \emph{tangent space} of $\M_r$ at $X$ is 
    \begin{equation*}
        T_X\M_r=(\mathcal{U}\otimes\mathcal{V})\oplus(\mathcal{U}\otimes\mathcal{V}^\perp)\oplus (\mathcal{U}^\perp\otimes\mathcal{V}).
    \end{equation*}
    \zzy{We use the abbreviation $T_X$ when the manifold $\M_r$ is clear from context.
    The projection operator onto the tangent space can be characterized as
    \begin{equation*}
        P_{T_{X}}(Y)=P_U\cdot Y + Y \cdot P_V- P_U\cdot Y \cdot P_V.
    \end{equation*}    }
\end{lemma}

\begin{definition}[Retraction]
    Let $X\in \M_r$ and $\xi \in T_X$. We define the natural retraction on $\M_r$ as
    \begin{align*}
        {R}(X+\xi) = \argmin_{\blue{Z \in \overline{\M_r}}} \|X+\xi-Z\|_F .
    \end{align*}
\end{definition}

\subsection{Existence of spurious critical points}
\label{sec:spurious}
Our study of spurious critical points is motivated by the study in \cite{hou2020fast}, which reveals a previously unknown property of the fixed rank matrix manifold $\M_r$. Namely, when minimizing the least squares loss function $f(Z) = \frac{1}{2}\|Z-X\|_F^2$ on $\M_r$ with $\text{rank} (X) = r$, there exist some points with rank smaller than $r$ which could also serve as the limit points of minimizing sequences. \zzy{This phenomenon was first reported in \cite{helmke1995critical} and later attracted more research interest. We summarize it in the following lemma. }

\begin{lemma}[{\cite[Lemma 3.8]{hou2020fast}}]
\label{lem: spurious_fp}
Consider using the Riemannian gradient descent algorithm
\begin{align*}
    Z_{k+1}={R}\left(Z_k - \alpha_k P_{T_{Z_k}}(\grad f(Z_k))\right)
\end{align*}
to minimize the least squares objective function
\begin{align*}
    f(Z)=\frac{1}{2}\|Z-X\|_F^2.
\end{align*}
Let the step size be $\alpha_k \equiv \alpha$. Assume $X = U_x D_x V_x^*$ is a singular value decomposition of  $X$, where $D \in \R^{r\times r}$ is a non-singular diagonal matrix and $U\in \F^{m \times r}$, $V\in\F^{n\times r}$. Then, 
\begin{enumerate}[label=\arabic*)]
\item  There are two types of fixed points: one is the ground truth $Z = X$, and the other is the set 
\begin{align*}
    \mathcal{S}_{\#} := &\big\{Z_\#: \, Z_\# = U_{1}D_{1} {V_{1}}^*, \text{ where } U_{1}, D_{1},  {V_{1}} \text{ are submatrices of }U_x, D_x, V_x, \text{ satisfying}\\
    & U_x = \left(U_{1}, U_{2}\right), \, D_x = \text{diag} \{D_{1}, D_{2}\}, \, V_x = \left(V_{1}, V_{2}\right) \text{ respectively, and } Z_\#\neq X\big\}.
\end{align*}
\item Specifically, if $X$ has distinct singular values, i.e. all the eigenvalues of $X$ have algebraic multiplicity equal to 1, then $\mathcal{S}_{\#}$ has cardinality $|\mathcal{S}_{\#}|=2^r-1$. Assume that $X=\sum_{i=1}^r d_i u_i v_i^*$, then $\mathcal{S}_{\#}=\{Z_\#=\sum_{i=1}^r d_i\eta_i u_i v_i^*, \text{ where }\eta\in\{0,1\}^r \text{ and } \eta\neq (1,\,1,\,\ldots,1)^*\}$.
\end{enumerate}
\end{lemma}

Motivated by the above lemma, we introduce the formal definition of the spurious critical points.

\begin{definition}[Spurious critical points]
\label{def:spurious}
    Assume that $X = U_xD_xV_x^*$ is a singular value decomposition of $X$. Then the set of \textbf{spurious critical points} with respect to $f(Z) =\frac{1}{2}\|Z-X\|_F^2$ on $\M_r$ is $\mathcal{S}_{\#} = \cup_{s=0}^{r-1} \mathcal{S}_s$,
    where each $\mathcal{S}_s$ can be characterized as
    \begin{align*}
        \mathcal{S}_s := &\big\{Z_\#: \, Z_\# \in \mathcal{S}_{\#}, \text{ rank}(Z_\#) =s \} \\
        = &\big\{Z_\#: \, Z_\# = U_{1}D_{1} {V_{1}}^*, \,\, U_{1} \in \F^{m\times s},\, V_{1} \in \F^{n\times s}, \,D_{1}\in\F^{s\times s}\}.
    \end{align*}
    Here $U = (U_1, U_2)$, $U_1 \in \F^{m\times s}$, $U_2 \in \F^{m\times (r-s)}$ is a block decomposition of $U$; similarly for $V$ and $D$.
\end{definition}

A simple example of a minimizing sequence converging to a spurious critical point instead of the ground truth $X$ has been given in Example \ref{ex: spurious}. Observing such phenomenon, one naturally asks how common this happens in practice. Interestingly, when the initial point $Z_0$ is sampled on $\M_r$ according to some general random sampling scheme, we observe that convergence to spurious critical points almost never happens. The goal of this work is thus to investigate the mechanism behind it. 

\subsection{Dynamical low-rank approximation}
\label{sec:DLRA}
The dynamical low-rank approximation was first proposed in \cite{koch2007dynamical} and soon gained popularity as a discretization method for the computation of low-rank evolution systems. It gives a neat description of the column space and core matrix of the low-rank matrix along the evolution. The decomposition enjoys better smoothness than SVD and other classical decompositions. \zzy{While a smooth version of SVD is only available when the gradient function is analytic, the dynamical low-rank approximation always preserves the smoothness of the gradient function. Thus,} it well suits our purpose.

\begin{lemma}[Dynamical low-rank approximation\footnote{\zzy{Strictly speaking, our ODE system is not an ``approximation'' but an exact characterization of the gradient flow. We stick to this terminology for ease of reference.}}, \cite{koch2007dynamical}]
\label{lem: koch2007dynamical}
Consider the gradient flow of a function $f(Z): \, \M_r \to \R$.
Assume $Z=USV^*$, where $U, \, V \in \F^{n \times r}$ are orthonormal, $S \in \R^{r \times r}$ \zzy{nonsingular}. \zzy{Let $M:= -\grad f(Z) = -P_{T_{Z}}(\nabla f(Z))$ denote the negative Riemannian gradient of $f(Z)$ on $\M_r$}. Impose the constraints $\dot{U}^* U = \dot{V}^* V =0$. Then the gradient flow of $f(Z)$ can be described by the following ODE system:
\begin{align}\label{eq:DLRA_original}
    \begin{cases}
        \dot{U}=P_U^{\perp}MVS^{-1},\\
        \dot{V}=P_V^{\perp}M^*U(S^{-1})^*,\\
        \dot{S}=U^*MV.
    \end{cases}
\end{align}
Here, $P_U^{\perp}=I-UU^*$ and $P_V^{\perp}=I-VV^*$.
\end{lemma}

The dynamical low-rank approximation introduces a  multiple-to-one mapping as a parameterization of $\M_r$.
Let $\St(n,r)$ denote the $n$ by $r$ Stiefel manifold, i.e. $\St(n,r) = \{U \in \F^{n \times r}: \, U^*U = I_r\}$. Then we have \zzy{that for $S$ nonsingular},
\begin{align*}
    \St(m,r) \oplus \St(n,r) \oplus \F^{r\times r} \quad &\to \quad \M_r \\
    (U,V,S)\quad &\mapsto \quad Z = USV^*.
\end{align*}
Since $S$ is not required to be diagonal, there are infinitely many tuples of $(U,V,S)$ corresponding to the same $Z$, and these tuples are not equivalent under permutations. However, after we impose the constraints $\dot{U}^* U = \dot{V}^* V =0$, from any initial tuple $(U_0,V_0,S_0)$ there is a \emph{unique} path in $\St(m,r) \oplus \St(n,r) \oplus \F^{r\times r}$ that describes the gradient flow of $f(Z)$ \zzy{using a similar argument as \cite{koch2007dynamical}}. In other words, as long as the initial decomposition $Z_0 = U_0S_0V_0^*$ is given, the decomposition that satisfies the dynamical low-rank relation is uniquely determined along the whole trajectory. 

The advantage of the dynamical low-rank approximation (\ref{eq:DLRA_original}) lies in the fact that the ODE system \zzy{generically stays continuous. This is especially remarkable for the singular vector matrices $U$ and $V$. As a comparison, SVD might enjoy uniqueness to some extent, but it is known to lose its differentiability when singular values coalesce \cite{dieci1999smooth}, and that could only be fixed with the unrealistic assumption of analyticity \cite{bunse1991numerical}. 
}

More specifically, in the SPSD setting, for the least squares function $f(Z) = \frac{1}{2}\|Z-X\|_F^2$, we have the following result.

\zzy{
\begin{lemma}[Existence of gradient flow]
\label{lemma:DLRA_existence}
    Consider the manifold of symmetric positive semi-definite (SPSD) matrices, i.e., $m=n$, and $\M_r = \left\{ Z \in \S_n, \, Z \succcurlyeq 0, \, \text{rank}(Z)=r \right\}$. Consider the least squares objective function $f(Z) = \frac{1}{2}\|Z-X\|_F^2$. Let $M:= -\grad f(Z)$ denote its negative Riemannian gradient. Let $Z_0 \in \M_r$ be the initialization of the gradient flow at time $T=0$, and $U_0 \in \St(n,r)$, $S_0 \in \S_r$ nonsingular such that $Z_0 = U_0 S_0 U_0^\top$. Then there exists a unique gradient flow satisfying
    \begin{align*}
        \begin{cases}
            \dot{U}=P_U^{\perp} M US^{-1},\\
            \dot{S}=U^*M U.
        \end{cases}
    \end{align*}
    for all $0 \le T < \infty$.
\end{lemma}

\begin{proof}
    The Riemannian gradient of the objective function $f(Z) = \frac{1}{2}\|Z-X\|_F^2$ is $P_{T_Z}(Z-X)$. Plugging in $M = -P_{T_Z}(Z-X)$, and noticing that $P_U^\perp Z = 0$ and $P_U U = U$, we get the following ODE system:
    \begin{align*}
            \begin{cases}
                \dot{U} = P_U^\perp X U S^{-1},  \\
                \dot{S} = -S + U^* X U.
            \end{cases}
    \end{align*}
    It suffices to show that the ODE system does not blow up in finite time. We prove that for any $T_1>0$, $\sigma_\text{min}(S)$ is bounded from below for all $T \in [0, T_1]$, where $\sigma_\text{min}(S)$ is the smallest eigenvalue of $S \in S_r$.
    
    At $T=0$, we have $\sigma_\text{min}(S)>0$. At a given time $T$, let the multiplicity of $\sigma_\text{min}(S)$ be $j$, i.e., $\sigma_{r-j}(S) > \sigma_{r-j+1}(S) = \ldots = \sigma_{r}(S)$. Denote $P_{U_{(r-j+1) \text{ to }r}}$ as the projection onto the corresponding eigen subspace. Using a similar argument as in \cite{magnus1985differentiating}, one can show that
    \begin{align*}
        \frac{\rd}{\rd t} \left(\sum_{l=r-j+1}^r \sigma_{l}(S) \right) = \text{tr}\left( P_{U_{(r-j+1) \text{ to }r}} \cdot \frac{\rd}{\rd t} S\right).
    \end{align*}
    In particular, when $\sigma_r(S)$ is a simple eigenvalue and $u_r$ is its eigenvector, this reduces to the classical result
    \begin{align*}
        \frac{\rd}{\rd t} \sigma_r(S) = u_r^* \left(\frac{\rd}{\rd t} S\right) u_r.
    \end{align*}
    Note that $\frac{\rd}{\rd t} S = -S + U^* X U$ and $X$ is positive semi-definite. Thus $\frac{\rd}{\rd t} S \succcurlyeq -S$, and we have
    \begin{align*}
        \frac{\rd}{\rd t} \left(\sum_{l=r-j+1}^r \sigma_{l}(S) \right) \ge  \text{tr}\left( P_{U_{(r-j+1) \text{ to }r}} \cdot (-S)\right) = -\sum_{l=r-j+1}^r \sigma_{l}(S).
    \end{align*}
    In particular, when $\sigma_r(S)$ is a simple eigenvalue, one has
    \begin{align*}
        \frac{\rd}{\rd t} \sigma_r(S) \ge -\sigma_r(S).
    \end{align*}
    By Gr\"onwall's inequality, $\sigma_{\text{min}}(S)$ decays no faster than exponentially fast. Thus it is bounded from below in any finite time interval.
\end{proof}
}

Under the above parameterization, any isolated critical point $Z_\#$ on $\M_r$ corresponds to a critical \zzy{set} on $\St(m,r) \oplus \St(n,r) \oplus \F^{r\times r}$ consisting of infinitely many points, denoted as $\N_{Z_\#}$:
\begin{align*}
    \N_{Z_\#} := \{(U_\#,V_\#,S_\#): \, U_\#S_\#V_\#^* = Z_\#\}.
\end{align*}
Some constraints need to be imposed on the above decomposition to make it a valid parameterization for a spurious critical point. We will discuss it in more detail in Section \ref{sec:parameterization}.

We do not distinguish between the parameterized gradient flow on $\St(m,r) \oplus \St(n,r) \oplus \F^{r\times r}$ and the original gradient flow on $\M_r$ when there is no confusion. To prove the asymptotic escape of spurious critical points on $\M_r$, then, is to prove the asymptotic escape of spurious critical submanifolds on $\St(m,r) \oplus \St(n,r) \oplus \F^{r\times r}$.

\subsection{Asymptotic escape of classical strict saddles}
\label{sec:asymptotic}
In this subsection, we introduce the classical results for the asymptotic escape of strict saddle points by gradient descent. Note that we only intend to include the results for the vanilla gradient descent. We do not cover the perturbed or stochastic gradient descent, as they are less relevant to our problem.

We emphasize that the spurious critical points in Definition \ref{def:spurious}, the subject of this study, are \emph{not} classical strict saddle points. It is because the Riemannian Hessian at the spurious critical points is {singular}, as will be revealed in subsequent sections. Therefore, the theorems and lemmas in this subsection are \emph{not} directly applicable to the spurious critical points. Nevertheless, these theorems and lemmas will be used in an indirect manner, on a rescaled system where the singularity is removed. 

The first theorem is a result on the stable and unstable manifolds of the gradient flow at a hyperbolic point.

\begin{theorem}[{\cite[The \emph{Center Manifold Theorem}]{perko2013differential}}]
\label{thm:escapeGF}
    Let $f \in C^r(E)$ where $E$ is an open subset of $\R^n$ containing the origin and $r\ge 1$. Let $x(t) = \phi_t(x_0)$ be the gradient flow of the system $\dot{x} = f(x)$. Suppose that $f(0)=0$ and that $Df(0)$ has $k$ eigenvalues with negative real part, $j$ eigenvalues with positive real part, and $m = n-k-j$ eigenvalues with zero real part. Then there exist
    \begin{enumerate}[label = (\arabic*)]
        \item A $k$-dimensional stable manifold $W^s(0)$ of class $C^r$ tangent to the stable subspace $E^s$ at 0, where for all $x_0 \in W^s(0)$,
        \begin{align*}
            \lim_{t\to +\infty}\phi_t(x_0) = 0;
        \end{align*}
        \item A $j$-dimensional unstable manifold $W^u(0)$ of class $C^r$ tangent to the unstable subspace $E^u$ at 0, where for all $x_0 \in W^u(0)$,
        \begin{align*}
            \lim_{t\to -\infty}\phi_t(x_0) = 0;
        \end{align*}
        \item And an $m$-dimensional center manifold $W^c(0)$ of class $C^r$ tangent to the center subspace $E^c$ at 0.
    \end{enumerate}
    Furthermore, $W^c(0)$, $W^s(0)$ and $W^u(0)$ are invariant under the gradient flow.
\end{theorem}

Next, we introduce the counterpart of the previous results for the gradient descent. Specifically, instead of $\dot{x} = f(x)$, we consider $f(x) = \varphi(x)$ where $\varphi(x)$ is the iteration function of the gradient descent algorithm. For example, when minimizing the least squares loss function on $\M_r$, the iteration function is $\varphi(Z) = R\left(Z-\alpha P_{T_{Z}}(Z-X)\right)$. The strict saddle point is defined as follows. It basically says that a strict saddle point is a hyperbolic point of the iteration function. 

\begin{definition}[Strict saddle point] 
\label{def:saddle2}
\zzy{Consider a function $f(\cdot):\,\M\rightarrow\R$ defined on a manifold $\M$}. We call $Z\in\M$ a \emph{strict saddle point} of $f$, if 
    \begin{enumerate}[label = (\arabic*)]
        \item $\text{grad} f(Z))=0$;
        \item $\text{Hess }f({Z})$ has at least one negative eigenvalue. 
    \end{enumerate}
\end{definition}

We then have the following theorem for the asymptotic escape of isolated saddle points.

\begin{theorem}[{\cite[Theorem 2.15]{hou2020analysis}}]
\label{thm:escape2}
    Let $f(\cdot):\M\rightarrow\R$ be a $C^2$ function on $\M$. Suppose that $f(\cdot):\M\rightarrow\R$ has either finitely many saddle points, or countably many saddle points in a compact submanifold of $\M$, and all saddle points of $f$ are strict saddles as is defined in Definition \ref{def:saddle2}. Let $\mathcal{A}$ denote the set of strict saddles. Then we have
    \begin{align*}
        \text{Prob}(\lim_{k\rightarrow\infty} Z_k\in \mathcal{A})=0
    \end{align*}
\end{theorem}

The proof of the theorem is based on {\cite[Theorem III.7]{shub2013global}}, which is very similar to Theorem \ref{thm:escapeGF} but focuses on the contraction/expansion of the iteration function. We omit the details here.

As is mentioned in the previous subsection, using the parameterization $\St(m,r) \oplus \St(n, r) \oplus \F^{r\times r} \to \M_r$, each single critical point $Z_\#$ corresponds to a submanifold $\N_{Z_\#}$. We need the following definitions of the analogy of strict saddle points for submanifolds.

\begin{definition}[Critical submanifold]
   For $f: \M\mapsto\R$, a connected submanifold $\mathcal{N} \subset \M$ is called a \emph{critical submanifold} of $f$ if every point $Z$ in $\mathcal{N}$ is a critical point of $f$, i.e. $\text{grad} f(Z) = 0$ for any $Z\in\N$.
\end{definition}

\begin{definition}[Strict critical submanifold]
\label{def:saddle3}
    A critical submanifold $\mathcal{N}$ of $f$ is called a \emph{strict critical submanifold}, if $\forall Z\in \mathcal{N}$,
    \begin{align*}
        \lambda_{\text{min}}(\text{Hess } f(Z)) \le c < 0,
    \end{align*}
    where $\lambda_{\text{min}}(\cdot)$ takes the smallest eigenvalue, and $c = c(\mathcal{N})$ is a uniform constant for all $Z\in\mathcal{N}$ depending only on $\mathcal{N}$.
\end{definition}

Using the above definitions, we have the following theorems on the asymptotic escape of strict critical submanifolds for gradient descent.

\begin{theorem}[{\cite[Theorem 2.19]{hou2020analysis}}]
\label{thm:escape3}
    Let $f(\cdot):\M\rightarrow\R$ be a $C^2$ function on $\M$. Suppose that $f(\cdot):\M\rightarrow\R$ has either finitely many critical submanifolds, or countably many critical submanifolds in a compact region of $\M$, and all of them are strict critical submanifolds as defined in Definition \ref{def:saddle3}. Let $\mathcal{A}$ denote the union of strict critical submanifolds. Then we have
    \begin{align*}
        \text{Prob}(\lim_{k\rightarrow\infty} Z_k\in \mathcal{A})=0.
    \end{align*}
\end{theorem}
We remark that the results on the asymptotic escape of  saddle points in the Euclidean space, e.g., the results in \cite{lee2016gradient}, can be seen as special cases of Theorem \ref{thm:escape3}.

\section{Main result}
\label{sec:main}
From this section on, we focus on the set of symmetric positive semi-definite (SPSD) matrices, i.e., $m=n$, and $\M_r = \left\{ Z \in \S_n, \, Z \succcurlyeq 0, \, \text{rank}(Z)=r \right\}$. The ground truth $X$ is also a rank-$r$ SPSD matrix.
Recall that by Definition \ref{def:spurious}, the set of spurious critical points is  
$    \mathcal{S}_{\#} = \cup_{s=0}^{r-1} \mathcal{S}_s
$,
where each $\mathcal{S}_s$ ($0 \le s \le r-1$) is the set of rank-$s$ spurious critical points, i.e.,
\begin{align*}
     \mathcal{S}_s := &\big\{Z_\#: \, Z_\# \in \mathcal{S}_{\#}, \text{ rank}(Z_\#) =s \} \\
     = &\big\{Z_\#: \, Z_\# = U_{1}D_{1} {U_{1}}^*, \,\, U_{1} \in \F^{n\times s}, \,D_{1}\in\F^{s\times s}\}.
\end{align*}
The first main result of this paper is as follows.
\begin{theorem}[Asymptotic escape of $\mathcal{S}_{r-1}$: gradient flow]
\label{thm:GF}
\blue{Let $f(Z) = \frac{1}{2} \|Z-X\|_F^2$, where $X \in \M_r$ has distinct eigenvalues. Let $Z_t: t\ge 0$ be the gradient flow of $f(Z)$ on $\M_r$ starting from a random initialization $Z_0$. Then we have that $Z_t \in \M_r \, \forall  \,\, 0\le t < +\infty$, and}
\begin{align*}
    \text{Prob }(\lim_{t\rightarrow\infty} Z_t\in S_{r-1})=0.
\end{align*}
\end{theorem}
The rest of this section is devoted to the proof of Theorem \ref{thm:GF}.

\subsection{Parameterization of \texorpdfstring{$\overline{\M_r}$}{Mr}}
\label{sec:parameterization}

In order to use the dynamical low-rank approximation from Section \ref{sec:DLRA}, we decompose a rank-$r$ matrix $Z \in \S_n$ into $Z = USU^*$, where $U \in \St(n,r)$ and $S \in \S_r$. This decomposition differs from the eigenvalue decomposition in that $S$ is not necessarily a diagonal matrix.

Consider a spurious critical point $Z_\# = U_{1} D_{1} {U_{1}}^{*} \in \M_s \subset \overline{\M_r}\backslash\M_r$, where $U_{1}\in \F^{n\times s}$ represents the $s$ eigenvectors that are also eigenvectors of $X$. We would like to determine a submanifold $\N_{Z_\#} \subset \St(n,r) \oplus \S_r$ that corresponds to $Z_\#$. Assume that
\begin{align*}
    Z_\# = U_\# S_\# U_\#^*,
\end{align*}
where $$S_\# = P_\# \Sigma_\# P_\#^*$$ is the eigenvalue decomposition of $S_\#$. Then there exists $U_3 \perp U_1$, such that
\begin{align*}
    U_\# = (U_{1}, U_{3})P_\#^*, \qquad \Sigma_\# = 
    \begin{pmatrix}
        D_{1} & 0 \\ 0 & 0
    \end{pmatrix}.
\end{align*}
In addition, for $Z_\#$ to be a critical point of $f(Z) = \frac{1}{2}\|Z-X\|_F^2$, we need $P_{T_{Z_\#}}(Z_\#-X) = 0$. One can show that this gives
\begin{align*}
    U_3 \perp U_x = (U_1, U_2).
\end{align*}
In other words, $U_3$, the $n\times (r-s)$ matrix that makes up for the missing rank, should be chosen to be perpendicular to the missing component $U_2$. This also gives us $\lim_{Z \to Z_\#}\text{grad} f(Z) = 0$, a property that will be useful in upcoming computations.

To sum up, a spurious critical point $Z_\# \in \mathcal{S}_{\#}$ can be parameterized as 
\begin{align*}
    \N_{Z_\#}= \left\{(U_\#, \, S_\#): \,\, U_\# = (U_{1}, U_{3})P_\#^*, \,\, S_\# = P_\#
    \begin{pmatrix}
        D_{1} & 0 \\ 0 & 0
    \end{pmatrix}
    P_\#^*, \,\, U_3 \perp U_x
    \right\},
\end{align*}
where $P_\# \in \text{SO}(r)$ is an orthonormal matrix.

\zzy{
\begin{lemma}
\label{lemma:submanifold}
    $\N_{Z_\#}$ is an embedded submanifold of the manifold $\M:=\St(n,r) \oplus \S_r$.
\end{lemma}

\begin{proof}
    See Appendix \ref{sec:proof_submanifold}.
\end{proof}
}

\subsection{Rescaled gradient flow}
\label{sec:rescaled}
Consider the dynamical low-rank description of the gradient flow for the objective function $f(Z) = \frac{1}{2}\|Z-X\|_F^2$. Impose the constraint $\dot{U}^* U = 0$ as required by Lemma \ref{lem: koch2007dynamical}. Plug the Riemannian gradient $P_{T_Z}(Z-X)$ into (\ref{eq:DLRA_original}), and notice that $P_U^\perp Z = 0$ and $P_U U = U$, we obtain the following ODE system:
\begin{align}
    \label{DLRA}
    \tag{DLRA}
        \begin{cases}
            \dot{U} = F(U,S) := P_U^\perp X U S^{-1},  \\
            \dot{S} = H(U,S) := -S + U^* X U.
        \end{cases}
\end{align}

The main tool for the proof of asymptotic escape is the following \emph{rescaled gradient flow} ODE system:
\begin{align}
    \label{DLRA*}
    \tag{DLRA*}
        \begin{cases}
            \dot{U} = \widetilde{F}(U,S) := P_U^\perp X U S^{-1} \cdot \sigma_{\text{min}}(S), \\
            \dot{S} = \widetilde{H}(U,S) := (-S + U^* X U)\cdot \sigma_{\text{min}}(S).
        \end{cases}
\end{align}
Here $\sigma_{\text{min}}(S)$ denotes the smallest eigenvalue of the $r \times r$ matrix $S$. In other words, the rescaled system (\ref{DLRA*}) is just the original system (\ref{DLRA}) times a scalar $\sigma_{\text{min}}(S)$.

We first show that the rescaled system (\ref{DLRA*}) is well-defined. 
\begin{lemma}[Continuity]
    The functions $ \widetilde{F}(U,S) $ and $ \widetilde{H}(U,S) $ are $C^0$ in {$\M_r$}.
\end{lemma}
\begin{proof}
    Inside $\M_r$, the matrix inverse $S^{-1}$ is well-defined, so are the functions $F(U,S) $ and $ H(U,S) $. Then use the fact that the smallest eigenvalue $\sigma_{\text{min}}(S)$ is $C^0$ with respect to $S$.
\end{proof}

\begin{lemma}[$C^0$-extension]
\label{lemma:C0extension}
    The functions $ \widetilde{F}(U,S) $ and $ \widetilde{H}(U,S) $ can be extended continuously to $\mathcal{S}_{r-1}$.
\end{lemma}
\begin{proof}
    Take any $Z_\# \in\mathcal{S}_{r-1}$ with parameterization $Z_\# = U_\#S_\#U_\#^*$. It suffices to show that $\lim_{Z\to Z_\#} \widetilde{F} (U,S)$ and $ \widetilde{H}(U,S) $ exist, and are independent of the specific choices of parameterization.
    
    Let $S = P \Sigma P^*$ and $S_\# = P_\# \Sigma_\# P_\#^*$ be the eigenvalue decompositions of $S$ and $S_\#$ respectively. Denote $p_i = P(:,i)$, and $p_{\#,i} = P_\#(:,i)$.
    Assume that $X=U_xD_xU_x^* = \sum_{i=1}^r d_i u_i u_i^*$. Since $Z_\# \in \mathcal{S}_{r-1}$, from the previous subsection, we know that
    \begin{align*}
        \Sigma_\# =
        \begin{pmatrix}
            D_{1} & 0 \\ 0 & 0
        \end{pmatrix},
    \end{align*}
    where $D_{1}$ is an $(r-1)\times(r-1)$ diagonal matrix, $D_{1} = \text{diag}\{d_{1}, \ldots, d_{r-1}\}$. Moreover, when $\|S-S_\#\|_F < \epsilon$ for small enough $\epsilon$, by the sin$\Theta$ theorem (Lemma \ref{lemma:sin_theta}), we have 
    \begin{align*}
        \Sigma = \text{diag}\{\sigma_1, \ldots, \sigma_{r-1}, \sigma_r\},
    \end{align*}
    where 
    \begin{align*}
        &\sigma_j > \text{min}\{d_{1}. \ldots, d_{r-1}\}-\epsilon, \quad 1\le j\le r-1;\\
        &0\le \sigma_r < \epsilon.
    \end{align*}
    In other words, $\sigma_r$ and the rest of the eigenvalues of $S$ are well-separated. Thus, when $\epsilon$ is small enough, we always have $\sigma_\text{min}(S) = \sigma_r$. 
    
    Consider $\varphi(S):= S^{-1}\sigma_\text{min}(S)$. When $\|S-S_\#\|_F < \epsilon$, we have 
    \begin{align*}
        \varphi(S) &= P \cdot \text{diag}\{\sigma_1^{-1}, \ldots, \sigma_{r-1}^{-1}, \sigma_r^{-1}\} \cdot  P^* \cdot \sigma_r \\
        &= P \cdot \text{diag}\left\{\frac{\sigma_r}{\sigma_1}, \ldots, \frac{\sigma_r}{\sigma_{r-1}}, 1\right\} \cdot P^* \\
        &= P \cdot \text{diag}\left\{\frac{\sigma_r}{\sigma_1}, \ldots, \frac{\sigma_r}{\sigma_{r-1}}, 0\right\} \cdot  P^* + p_r p_r^*.
    \end{align*}
    Thus,
    \begin{align*}
        \lim_{S\to S_\#}\varphi(S) &= \lim_{S\to S_\#} \left(P \cdot \text{diag}\left\{\frac{\sigma_r}{\sigma_1}, \ldots, \frac{\sigma_r}{\sigma_{r-1}}, 0\right\} \cdot P^* + p_r p_r^* \right )\\
        &= 0 \,\, + \,\, p_{\#,r} p_{\#,r}^* \\
        &= p_{\#,r} p_{\#,r}^*.
    \end{align*}
    In other words, $\varphi(S)$ can be continuously extended to $S_\#$.
    
    We can now compute the limits of $\widetilde{F}$ and $\widetilde{H}$. Note that 
    \begin{align*}
        \widetilde{F}(U,S) &= P_U^\perp X U \cdot \varphi(S).
    \end{align*}
    Using the parameterization  
    \begin{align*}
        Z_\# = U_\#S_\#U_\#^*: \quad U_\# = (U_{1}, U_{3})P_\#^*,  \quad S_\# = P_\#
        \begin{pmatrix}
            D_{1} & 0 \\ 0 & 0
        \end{pmatrix}
        P_\#^*,
    \end{align*}
    we have
    \begin{align*}
        \lim_{(U,S) \to (U_\#,S_\#)} \widetilde{F}(U,S) &= P_{U_\#}^\perp X U_\#\cdot \lim_{S\to S_\#} \varphi(S) \\
        &= P_{U_\#}^\perp X U_\#\cdot p_{\#,r} p_{\#,r}^* \\
        &= (I - P_{U_1} - P_{U_3}) \cdot (U_1 D_1 U_1^* + U_2 D_2 U_2^*) \cdot (U_{1}, U_3)P_\#^* \cdot p_{\#,r} p_{\#,r}^* \\
        &=U_2 D_2 U_2^* \cdot (U_{1}, U_3) \cdot P_\#^* \cdot p_{\#,r} p_{\#,r}^* \\
        &= 0
    \end{align*}
    As for $\widetilde{H}(U,S)$, since $H(U,S)$ is bounded and $\sigma_\text{min}(S)$ converges to zero, we have
    \begin{align*}
        \lim_{(U,S) \to (U_\#,S_\#)} \widetilde{H}(U,S) = \lim_{(U,S) \to (U_\#,S_\#)} H(U,V,S)\cdot \sigma_{\text{min}}(S) = 0.
    \end{align*}
    Thus,  $ \widetilde{F}(U,S)$ and $ \widetilde{H}(U,S) $ can both be extended continuously to $\mathcal{S}_{r-1}$, independent of the parameterization.
\end{proof}

\subsection{Critical points of the rescaled system}
\label{sec:critical_rescaled}
In this section, we show that the ODE systems (\ref{DLRA}) and (\ref{DLRA*}) have the same critical points.

\zzy{
\begin{lemma}[Existence of rescaled gradient flow]
    Consider the rescaled ODE system (\ref{DLRA*}). Let $Z_0 \in \M_r$ be the initialization of the gradient flow at time $T=0$, and $U_0 \in \St(n,r)$, $S_0 \in \S_r$ nonsingular such that $Z_0 = U_0 S_0 U_0^\top$. Then there exists a unique gradient flow that satisfies  (\ref{DLRA*}) for all $T \in [0,\infty)$.
\end{lemma}

\begin{proof}
    The proof follows the same idea as that of Lemma \ref{lemma:DLRA_existence}. {We show that within finite time, $(U,S)$ remains in a region where $\widetilde{F}$ and $\widetilde{H}$ are Lipschitz continuous.} Note that $\nabla_{S_{ij}}(S^{-1}) = -S^{-1} E_{ij} S^{-1}$ where $E_{ij}$ is the indicator matrix of the $(i,j)$-entry. Note also that the smallest eigenvalue $\sigma_\text{min}(S)$ is Lipschitz continuous with respect to $S$ \cite{kangal2018subspace}. Thus the Lipschitz continuity of $\widetilde{F}$ and $\widetilde{H}$ holds if $S^{-1}$ is bounded. This is true if $\sigma_\text{min}(S)$ is bounded from below. 
    
    At a given time $T$, let the multiplicity of $\sigma_\text{min}(S)$ be $j$, i.e., $\sigma_{r-j}(S) > \sigma_{r-j+1}(S) = \ldots = \sigma_{r}(S)$. Denote $P_{U_{(r-j+1) \text{ to }r}}$ as the projection onto the corresponding eigen subspace. Using a similar argument as in \cite{magnus1985differentiating}, we now have
    \begin{align*}
        \frac{\rd}{\rd t} \left(\sum_{l=r-j+1}^r \sigma_{l}(S) \right) &= \text{tr}\left( P_{U_{(r-j+1) \text{ to }r}} \cdot \frac{\rd}{\rd t} S\right) \\
        &= \text{tr}\left( P_{U_{(r-j+1) \text{ to }r}} \cdot (-S + U^* X U)\cdot \sigma_{\text{min}}(S) \right) \\
        &\ge \text{tr}\left( P_{U_{(r-j+1) \text{ to }r}} \cdot (-S) \right) \cdot \sigma_{\text{min}}(S)\\
        &= -\left(\sum_{l=r-j+1}^r \sigma_{l}(S)\right)\cdot \sigma_{\text{min}}(S).
    \end{align*}
    In particular, when $\sigma_r(S)$ is a simple eigenvalue, this reduces to
    \begin{align*}
        \frac{\rd}{\rd t} \sigma_r(S) \ge -\sigma_r(S)^2.
    \end{align*}
    Thus $\sigma_{\text{min}}(S)$ decays no faster than geometrically due to Gr\"onwall's inequality. Thus it is bounded from below in any finite time interval.
\end{proof}
}

\begin{lemma}[Limit points]
\label{Lemma:limit}
    Let $Z_0\in\M_r$. Then the critical points of the ODE system (\ref{DLRA*}) are the same as those of (\ref{DLRA}). Moreover, the gradient flows starting from the same initial point always converge to the same limit point.
\end{lemma}
\zzy{
\begin{proof}
    Observe that the rescaled system (\ref{DLRA*}) is just the original system (\ref{DLRA}) multiplied by a scalar: 
    \begin{align*}
        \widetilde{F}(U,S) &= F(U,S) \cdot \sigma_\text{min}(S), \\
        \widetilde{H}(U,S) &= H(U,S) \cdot \sigma_\text{min}(S).
    \end{align*}
    Thus the gradient flow of the rescaled system follows the same path as the original system. In other words, let $Z_t \,(t\ge 0)$ and $\widetilde{Z}_t \,(t\ge 0)$ be the solutions of (\ref{DLRA}) and (\ref{DLRA*}) starting from the same initial point $Z_0$, then for any time $t\ge 0$, there exists a corresponding time $w \ge 0$ such that $\widetilde{Z}_t = Z_w$. 
    
    When the time goes to infinity, both flows have limit points because both are minimizing flows of a coercive and lower-bounded function $f(Z)$. Denote them as $\widetilde{Z}_\infty$ and $Z_\infty$ respectively. Then either $\widetilde{Z}_\infty = Z_\infty$, or there exists a finite $T$ such that $\widetilde{Z}_\infty = Z_T$. 
    
    We now argue that only $\widetilde{Z}_\infty = Z_\infty$ is possible.
    Looking at the ODE system (\ref{DLRA*}), a critical point has to satisfy either $F(U,S)=H(U,S)=0$, or $\sigma_\text{min}(S) = 0$. In the former case,  $F(U,S)=H(U,S)=0$ means such $(U,S)$ is stationary for (\ref{DLRA}), so $\widetilde{Z}_\infty = Z_\infty$. In the latter case, such $(U,S)$ has to be $Z_\infty$ because we know that $\sigma_\text{min}(S)$ cannot be zero at any finite time from Lemma \ref{lemma:DLRA_existence}.
    So either way, $\widetilde{Z}_\infty = Z_\infty$. Therefore, the critical points of (\ref{DLRA*}) could only be those of (\ref{DLRA}). 
\end{proof}
}

By Lemma \ref{Lemma:limit}, if we can prove that gradient flows of (\ref{DLRA*}) starting from random initializations almost surely avoids the spurious critical points in $\mathcal{S}_{r-1}$, we immediately have that the same results apply to (\ref{DLRA}). In the next subsection, we will show that this is much easier to prove for the rescaled system than for the original system, because the points in $\mathcal{S}_{r-1}$ are now {strict saddle points in the classical sense}.

\subsection{Landscape around the critical points}
\label{sec:landscape}
We now analyze the landscape around the critical points of (\ref{DLRA*}).
In fact, we will show that the $C^0$-extension that we proved in Lemma \ref{lemma:C0extension} can be improved to a $C^1$-extension.

\begin{lemma}[$C^1$-extension]
\label{lemma:C1extension}
    Assume that the eigenvalues of the ground truth matrix $X$ are all distinct. The functions $ \widetilde{F}(U,S) $ and  $ \widetilde{H}(U,S) $ can be $C^1$-extended to $\mathcal{S}_{r-1}$.
\end{lemma}

\begin{proof}

    We first compute $\nabla F$ and $\nabla H$ in the interior of $\M_r$. In this region, $S$ is non-singular, and all the derivatives are well defined. Let $\xi = (\xi_1, \xi_2)$ be a placeholder for the directional derivative, where $\xi_1$ and $\xi_2$ correspond to the direction of $U$ and $S$ respectively. Direct computation gives 
    \begin{align*}
        \nabla F(U,S) [\xi] & = 
        \begin{pmatrix}
            -(U\xi_1^* + \xi_1 U^*) X U S^{-1} + P_U^\perp X \xi_1 (S^{-1})^* \\
            - P_U^\perp X U  S^{-1} \xi_2 S^{-1}
        \end{pmatrix}, \\
        \nabla H(U,S)[\xi] & = 
        \begin{pmatrix}
            \xi_1^* X U +  U^* X \xi_1 \\
            -\xi_2
        \end{pmatrix}.
    \end{align*}
    To extend $\nabla F$ and $\nabla H$ themselves to $\mathcal{S}_{r-1}$ is impossible: $S^{-1}$ is singular near $\mathcal{S}_{r-1}$, causing the derivatives to explode. We aim to show that it becomes possible with the rescaled system (\ref{DLRA*}). 
    
    For this purpose, we define the following function, which is the directional derivative of $\varphi(S)$ along the direction $\eta$:
    \begin{align*}
        \psi(S,\eta):=\nabla \varphi(S)[\eta] = \nabla (S^{-1}\sigma_\text{min}(S))[\eta].
    \end{align*}
    We follow the same notations as before. Direct computation gives
    \begin{align*}
        \lim_{S\to S_\#}\psi(S,\eta) =\lim_{S\to S_\#} \left( -S^{-1} \eta S^{-1} \sigma_\text{min}(S) + S^{-1} \cdot \nabla \sigma_\text{min}(S)[\eta] \right).
    \end{align*}
    We know from the proof of Lemma \ref{lemma:C0extension} that when $\|S-S_\#\|_F< \epsilon$ for small enough $\epsilon$, the larger eigenvalues $\sigma_1$ to $\sigma_{r-1}$ and the smallest eigenvalue $\sigma_r$ are well-separated. In fact, assuming that the eigenvalues of $X$ are distinct, for small enough $\epsilon$, all the eigenvalues of $S$ are well-separated, and the corresponding eigenvectors are continuous with respect to the change of $S$. In this case, we know from \cite{magnus1985differentiating} that 
    \begin{align*}
        \nabla \sigma_r(S)[\eta] = p_r^* \eta p_r.
    \end{align*}
    Thus, we have 
    \begin{align*}
        \lim_{S\to S_\#}\psi(S,\eta) =\lim_{S\to S_\#} \left( -S^{-1} \eta S^{-1} \sigma_r + S^{-1} p_r^* \eta p_r \right).
    \end{align*}
    For simplicity, we focus on the case $\F = \R$. Since $\{p_i p_j^*\}_{i,j = 1}^r$ form a complete orthogonal basis of $\R^{r\times r}$, we can write
    \begin{align*}
        \eta = \sum_{1\le i,j\le r} c_{ij} p_i p_j^*.
    \end{align*}
    Such decomposition is continuous around $S_\#$, since all $\sigma_i$'s are well-separated and all $p_i$'s are continuous with respect to the change of $S$. 
    
    It now suffices to compute $\lim_{S\to S_\#}\psi(S,\eta)$ for $\eta = p_i p_j^*$, as $\psi(S,\eta)$ is linear in $\eta$. This comes in the following cases:
    \begin{enumerate}[label = (\arabic*)]
        \item If $i,\,j<r$:
        \begin{align*}
            \lim_{S\to S_\#}\psi(S,p_i p_j^*) &= \lim_{S\to S_\#} \left( -S^{-1} p_i p_j^* S^{-1} \sigma_r + S^{-1} p_r^* p_i p_j^* p_r \right) \\
            &= \lim_{S\to S_\#} \left( -P \Sigma^{-1}e_i e_j^* \Sigma^{-1}\sigma_r P^* + S^{-1} \cdot 0 \right) \\
            &= \lim_{S\to S_\#} \left(-P \cdot 0 \cdot P^* + 0 \right)\\
            & = 0. 
        \end{align*}
        \item If $i<r$, $j=r$:
        \begin{align*}
            \lim_{S\to S_\#}\psi(S,p_i p_r^*) &= \lim_{S\to S_\#} \left( -S^{-1} p_i p_r^* S^{-1} \sigma_r + S^{-1} p_r^* p_i p_r^* p_r \right)\\
            &= \lim_{S\to S_\#} \left( -P \Sigma^{-1}e_i e_r^* \Sigma^{-1}\sigma_r P^* + S^{-1} \cdot 0 \right) \\
            &= d_i^{-1} p_i p_r^*.
        \end{align*}
        \item If $i=r$, $j<r$:
        Similar to the previous case,
        \begin{align*}
             \lim_{S\to S_\#}\psi(S, p_r p_j^*) = d_j^{-1} p_r p_j^*.
        \end{align*}
        \item If $i=j=r$:
        \begin{align*}
            \lim_{S\to S_\#}\psi(S,p_r p_j^*) &= \lim_{S\to S_\#} \left( -S^{-1} p_r p_r^* S^{-1} \sigma_r + S^{-1} p_r^* p_r p_r^* v_r \right) \\
            &=  \lim_{S\to S_\#}\left( -P \Sigma^{-1} e_r e_r^* \Sigma^{-1} \sigma_r  P^* + S^{-1}\right) \\
            &= \lim_{S\to S_\#}\left(P \cdot \text{diag}\left\{\sigma_1^{-1}, \ldots, \sigma_{r-1}^{-1},-\sigma_r^{-1}+\sigma_r^{-1}\right\} \cdot P^*\right) \\
            &= \lim_{S\to S_\#}\left(P \cdot \text{diag}\left\{\sigma_1^{-1}, \ldots, \sigma_{r-1}^{-1}, 0 \right\} \cdot P^*\right) \\
            &= P_\# \cdot \text{diag}\left\{\sigma_1^{-1}, \ldots, \sigma_{r-1}^{-1},0\right\} \cdot P_\#^*.
        \end{align*}
    \end{enumerate}
    Therefore, $\psi(S,\eta)$ can be continuously extended to $S_\#$ for any $\eta$. 
    
    We now compute the derivatives of $\widetilde{F}$ and $\widetilde{H}$ at $Z_\#$. The directional derivative in $U$ only involves $\varphi(S_\#)$, and we have
    \begin{align*}
        \lim_{Z\to Z_\#}\nabla_U \widetilde{F}(U,S)[\xi_1] &= \lim_{Z\to Z_\#} \left(-(U\xi_1^* + \xi_1 U^*) X U S^{-1}  \sigma_\text{min}(S) + P_U^\perp X \xi_1 S^{-1}  \sigma_\text{min}(S) \right)\\
        &= -(U_\# \xi_1^*  + \xi_1 (U_\#)^*)XU_\# \cdot \varphi(S_\#) + P_{U_\#}^\perp X \xi_1 \cdot \varphi(S_\#) \\
        &= -(U_\# \xi_1^*  + \xi_1 (U_\#)^*)XU_\# \cdot p_r p_r^* + P_{U_\#}^\perp X \xi_1 p_r p_r^*.
    \end{align*}
    As for the directional derivative in $S$, we now make use of $\psi(S_\#,\eta)$:
    \begin{align*}
        \lim_{Z\to Z_\#} \nabla_S \widetilde{F}(U,S)[\xi_2] &= \lim_{Z\to Z_\#} \nabla_S \left(P_U^\perp X U S^{-1} \right)[\xi_2] \\ 
        &= \lim_{Z\to Z_\#} \left( P_U^\perp X U \psi(S,\xi_2) \right) \\
        &= P_{U_\#}^\perp X U_\# \cdot \psi(S_\#,\xi_2).
    \end{align*}
    Since $P_{U_\#}^\perp X U_\#= 0$ and $\psi(S_\#,\xi_2)$ is bounded, we have 
    \begin{align*}
        \lim_{Z\to Z_\#} \nabla_S \widetilde{F}(U,S)[\xi_2] = 0 \cdot \psi(S_\#,\xi_2) = 0.
    \end{align*}
    Thus, the derivatives of $\widetilde{F}$ can be extended continuously to $Z_\#$, and we have 
    \begin{align*}
        \lim_{Z\to Z_\#} \nabla \widetilde{F}(U,S) &= 
        \begin{pmatrix}
            -(U_\# \xi_1^*  + \xi_1 (U_\#)^*)XU_\# \cdot p_r p_r^* + P_{U_\#}^\perp X \xi_1 \cdot p_r p_r^* \\
            0
        \end{pmatrix}.
    \end{align*}
    As for the derivative of $\widetilde{H}$, we have 
    \begin{align*}
        \lim_{Z\to Z_\#} \nabla \widetilde{H}(U,S)[\xi] = \lim_{Z\to Z_\#} 
        \begin{pmatrix}
            (\xi_1^* X U + U^*X\xi_1) \cdot \sigma_\text{min}(S) \\
            -\xi_2 \cdot \sigma_\text{min}(S) + (-S + U^* X U) \nabla_S \sigma_\text{min}(S)[\xi_2]
        \end{pmatrix}= \begin{pmatrix}
            0\\
            0
        \end{pmatrix}.
    \end{align*}
    Thus, we have shown that the derivatives of  $ \widetilde{F}(U,S)$ and $\widetilde{H}(U,S) $ can both be extended continuously to such $Z_\#$, which is equivalent to saying that the functions themselves can be $C^1$-extended to such $Z_\#$.
\end{proof}

The $C^1$-extension is crucial to the landscape analysis of the system (\ref{DLRA*}) at the submanifolds corresponding to the rank-$(r-1)$ spurious critical points. It enables us to compute the Jacobian right at those submanifolds, and determine its eigenvalues. We now show that those submanifolds are actually strict critical submanifolds of the system (\ref{DLRA*}).

\begin{lemma}[Strict critical submanifold]
\label{lemma:strict}
    Assume that the eigenvalues of the ground truth matrix $X$ are all distinct. Given a point $Z_\#\in\mathcal{S}_{r-1}$, let $\N_{Z_\#} = \{(U_\#,S_\#): \, U_\#S_\#U_\#^* = Z_\#\}$ be the submanifold after parameterization that corresponds to $Z_\#$. Then  $\N_{Z_\#}$ is a strict critical submanifold of the system (\ref{DLRA*}). 
\end{lemma}

\begin{proof}
    The goal is to show that for any $(U_\#,S_\#) \in \N_{Z_\#}$, it is a hyperbolic point of the gradient flow with at least one escape direction, and all these points in $\N_{Z_\#}$ share a common escape direction perpendicular to the submanifold itself with a uniformly bounded eigenvalue.
    We will determine this escape direction by construction, using the results from the proof of Lemma \ref{lemma:C1extension}.
    
    Recall that $S = P \Sigma P^*$, $S_\# = P_\# \Sigma_\# P_\#^*$, and $X=U_xD_xU_x^* = \sum_{i=1}^r d_i u_i u_i^*$. Let 
    \begin{align*}
        \xi = (\xi_1, \, \xi_2), \quad \xi_1 = u_r p_{\#,r}^*,\quad \xi_2 = 0.
    \end{align*}
    Note that $X = U_1 D_1 U_1^* + U_2 D_2 U_2^*$, where $U_2 = u_r$, $D_2 = d_r$, and $U_2 D_2 U_2^*$ is the missing component in this spurious critical point $Z_\#$. In other words, we construct $\xi$ exactly along the direction of this missing component. Using this property, we have that
    \begin{align*}
        &\nabla_U\widetilde{F}(U,S)[\xi_1] \mid_{(U,S)=(U_\#,S_\#)} \\
        &= -(U_\# \xi_1^*  + \xi_1 U_\#^*)XU_\# \cdot p_{\#,r} p_{\#,r}^* + P_{U_\#}^\perp X \xi_1 \cdot p_{\#,r} p_{\#,r}^* \\
        &=  -(U_\# p_{\#,r} u_r^*  + u_r p_{\#,r}^* U_\#^*)X U_\#  \cdot p_{\#,r} p_{\#,r}^* + (I - P_{U_1}^\perp - P_{U_3}^\perp) X\cdot u_r p_{\#,r}^* \cdot p_{\#,r} p_{\#,r}^* \\
        &= -\left((0,U_3)u_r^* + u_r (0,U_3)^*\right) \left(U_1 D_1 U_1^* + U_2 D_2 U_2^* \right) \cdot U_\# p_{\#,r} p_{\#,r}^* + U_2 D_2 U_2^* \cdot u_r p_{\#,r}^* \\
        &=0 + d_r u_r u_r^* \cdot u_r p_{\#,r}^*\\
        &= d_r u_r p_{\#,r}^*,
    \end{align*}
    and
    \begin{align*}
        \nabla_S\widetilde{F}(U,S)[\xi_2] \mid_{(U,S)=(U_\#,S_\#)}  = 0.
    \end{align*}
    Thus,
    \begin{align*}
        \nabla \widetilde{F}(U,S)[\xi] \mid_{(U,S)=(U_\#,S_\#)} = d_r u_r p_{\#,r}^* + 0 = d_r u_r p_{\#,r}^*.
    \end{align*}
    Meanwhile,
    \begin{align*}
         \nabla \widetilde{H}(U,S)[\xi] \mid_{(U,S)=(U_\#,S_\#)} = 0.
    \end{align*}
    Putting everything together, we have
    \begin{align*}
        \nabla(\widetilde{F}, \widetilde{H})[\xi] &= d_r \cdot (u_r p_{\#,r}^*, 0) \\
        &= d_r \cdot \xi.
    \end{align*}
    This means that $\xi = (u_r p_{\#,r}^*, \,0)$ is an eigenvector of the Jacobian $\nabla(\widetilde{F}, \widetilde{H})$ with eigenvalue $d_r$, which is positive. 
    
    Thus, for every tuple $(U_\#,S_\#)$ in $\N_{Z_\#}$, we have found an escape direction with uniform eigenvalue. So $\N_{Z_\#}$ is a strict critical submanifold as desired.
\end{proof}

\subsection{Proof of the main result}
We now prove Theorem \ref{thm:GF} using the results from previous subsections.

\begin{proof}[\textbf{Proof of Theorem \ref{thm:GF}}]
\label{sec:proof_main}
    By Lemma \ref{lem: spurious_fp}, there are only finitely many spurious critical points in $\mathcal{S}_{r-1}$. By Lemma \ref{lemma:strict}, for each $Z_\# \in \mathcal{S}_{r-1}$, in the parameterized domain $\St(n,r)\oplus \S_{r}$, the corresponding submanifold $\N_{Z_\#}$ is a strict critical submanifold for the rescaled gradient flow. Since there are only finitely many of them, we can apply Theorem \ref{thm:escapeGF}. This implies that the rescaled gradient flow in the parameterized domain almost never converges to $\cup_{Z_\# \in \mathcal{S}_{r-1}}\N_{Z_\#}$. Thus the rescaled gradient flow in the original domain $\M_r$ also almost never converges to $\mathcal{S}_{r-1}$. By Lemma \ref{Lemma:limit}, the original gradient flow has the same limit as the rescaled gradient flow. Thus the original gradient flow enjoys the same result, i.e.,
    $\text{Prob }(\lim_{t\rightarrow\infty} Z_t\in S_{r-1})=0$.
\end{proof}

\section{\blue{Main result} for the gradient descent}
\label{sec:GD}
The previous section has focused on the gradient flow. In this section we derive the result for the gradient descent, namely the asymptotic escape of the Riemannian gradient descent algorithm from the spurious critical points in $\mathcal{S}_{r-1}$.

\begin{lemma}[\blue{Asymptotic escape of $\mathcal{S}_{r-1}$: gradient descent}]
\label{lemma:vary}
    \blue{Let $\M_r$ be the rank-$r$ SPSD matrix manifold. Consider $f(Z) = \frac{1}{2} \|Z-X\|_F^2$ where $X \in \M_r$ has distinct eigenvalues. Let $Z_0 \in \M_r$ be a random initialization, and $\{Z_k\}_{k=0}^{\infty}$ be the sequence generated by the following Riemannian gradient descent algorithm with varying step size:
    \begin{align}
    \label{eq:vary}
        Z_{k+1} = R\left(Z_k-\alpha\cdot \sigma_r(Z_k) \cdot P_{T_{Z_k}}\big(\nabla f(Z_k)\big)\right),
    \end{align}
    i.e. $\alpha_k = \alpha \cdot \sigma_r(Z_k)$, where $\sigma_r(Z_k)$ is the $r$-th eigenvalue of $Z_k$, and $\alpha>0$. Assume that $Z_k \in \M_r$ for any $k<+\infty$, i.e., the sequence stays inside $\M_r$ at any finite step. Then we have 
    \begin{align*}
        \text{Prob }(\lim_{k\rightarrow\infty} Z_k\in S_{r-1})=0.
    \end{align*}
    In particular, this holds true for arbitrarily large $\alpha>0$.}
\end{lemma}

\begin{remark}
    A few remarks are in order.
    \begin{enumerate}[label=(\arabic*)]
        \item The stepsize $\alpha_k = \alpha \cdot \sigma_r(Z_k)$ is varying but not necessarily diminishing. It is important to note that there is no upper bound on the constant $\alpha$. Thus even though $\sigma_r(Z_k) \to 0$ as $Z_k \to Z_\#$, the constant $\alpha$ can be chosen accordingly so that $\{\alpha_k\}$ can be arbitrarily close to non-diminishing stepsize.  
        \item The reason for the choice $\alpha_k = \alpha \cdot \sigma_r(Z_k)$ is similar to the rescaling of the ODE system (\ref{DLRA*}) in the previous section. Namely, this makes the Jacobian of the iteration function $C^1$-extendable to the rank-$(r-1)$ spurious critical points in $\mathcal{S}_{r-1}$, using the same techniques as in the proof of Lemma \ref{lemma:C1extension}.
    \end{enumerate}
\end{remark}

\begin{proof}[Proof of Lemma \ref{lemma:vary}]
    We use the same notations as before, namely $Z = USU^*$, $S = P \Sigma P^*$, $S_\# = P_\# \Sigma_\# P_\#^*$, and $X=U_xD_xU_x^* = \sum_{i=1}^r d_i u_i u_i^*$. We also let $Z = U_z \Sigma U_z^*$ denote the SVD of $Z$, which implies $U_z = U \cdot P^*$. We let $\widetilde{U} \in \St(n,n-r)$ be the orthogonal complement of $U$. It is also the orthogonal complement of $U_z$. Since $U = (U_1, U_3)$, where $U_3 \perp U_2$, we know that $\text{span}\{U_2\} \subset \text{span}\{\widetilde{U}\}$. Without loss of generality, we let $U_2$ be the first column of $\widetilde{U}$.
    
    Consider the iteration function
    \begin{align}\label{eq:iter}
        \begin{split}
            \phi(Z) &= R\left(Z-\alpha\cdot \sigma_r(Z) \cdot P_{T_{Z}}(\nabla f(Z))\right) \\
            &= R\left(Z-\alpha\cdot \sigma_r(Z) \cdot \text{grad}f(Z) \right).
        \end{split}
    \end{align}
    Here $\text{grad}f(Z)$ is the Riemannian gradient. 
    The Jacobian of the iteration function is 
    \begin{align*}
        D\phi(Z) = I - \alpha \cdot \left( \sigma_r(Z) \cdot \text{Hess}f(Z) + D \sigma_r(Z)\cdot \text{grad}f(Z) \right).
    \end{align*}
    It has been shown in \cite{vandereycken2013low} that 
    \begin{align}\label{eq:Hess}
        \text{Hess}f(Z)[\xi] = \xi + P_{U_z}^\perp (Z-X) \widetilde{U}N \Sigma^{-1} U_z^* + U_z \Sigma^{-1} N^* \widetilde{U}^* (Z-X) P_{U_z}^\perp,
    \end{align}
    where the vector $\xi$ is parameterized as 
    \begin{align*}
        \xi = U_z M U_z^* + U_z N\widetilde{U}^* + \widetilde{U} N^* U_z^*, \quad M\in \F^{r\times r}, \quad N \in \F^{r \times (n-r)}.
    \end{align*}
    In particular, when $\F = \R$, the degree of freedom of $\xi$ is $\frac{r(2n-r+1)}{2}$. It is equal to the dimension of the tangent space that $\xi$ lies in, which is the same as the dimension of the manifold.
    
    Consider $\lim_{Z\to Z_\#} D\phi(Z)$ for $Z_\# \in \mathcal{S}_{r-1}$. Note that the parameterization from Section \ref{sec:parameterization} ensures that $\text{span}\{U_{3}\}\perp\text{span}\{U_{1},U_{2}\}$, so that $Z_\#$ is a valid critical point, i.e. $\text{grad}f(Z_\#) = 0$. Plugging Equation (\ref{eq:Hess}) into Equation (\ref{eq:iter}), we have
    \begin{align*}
        D\phi (Z_\#)[\xi]&:= \lim_{Z\to Z_\#}D\phi (Z)[\xi] \\
        &= \xi - \alpha \cdot \left(\lim_{Z\to Z_\#}(\sigma_r(Z) \cdot \text{Hess}f(Z)[\xi])+ D \sigma_r(Z)[\xi] \cdot \text{grad}f(Z_\#)\right) \\
        &= \xi - \alpha \cdot \left(\lim_{Z\to Z_\#}\left(\sigma_r(Z) \cdot \text{Hess}f(Z) [\xi]\right)\right) \\
        &= \xi - \alpha \cdot \left(\lim_{Z\to Z_\#}\left(\sigma_r(Z)  \cdot \xi - P_U^\perp (Z-X) \widetilde{U} N \Sigma^{-1} U^* - U \Sigma^{-1} N^* \widetilde{U}^* (Z-X) P_U^\perp \right)\right) \\
        &= \xi - \alpha \cdot \left(0 \cdot \xi - U_{2} D_{2} {U_{2}}^{\top} \widetilde{U}N_2 \left(\lim_{\Sigma\to\Sigma_\#}\Sigma^{-1}\sigma_r\right) U_\#^* - U_\# \left(\lim_{\Sigma\to\Sigma_\#}\Sigma^{-1}\sigma_r\right)N^* \widetilde{U}^* U_{2} D_{2} {U_{2}}^{\top} \right) \\
        &= \xi + \alpha \cdot \left( U_{2} D_{2} {U_{2}}^{\top} \widetilde{U}N_2 \left(\lim_{\Sigma\to\Sigma_\#}\Sigma^{-1}\sigma_r\right) U_\#^* + U_\# \left(\lim_{\Sigma\to\Sigma_\#}\Sigma^{-1}\sigma_r\right)N^* \widetilde{U}^* U_{2} D_{2} {U_{2}}^{\top} \right).
    \end{align*}
    Here, similar to the proof of Lemma \ref{lemma:C0extension}, we have
    \begin{align*}
        \lim_{\Sigma\to\Sigma_\#}\Sigma^{-1}\sigma_r = \text{diag}\{0,\ldots, 0,1\} = e_r e_r^*.
    \end{align*}
    Thus, it follows that
    \begin{align*}
        D\phi (Z_\#)[\xi] = \xi + \alpha \cdot \left( U_{2} D_{2} {U_{2}}^{\top} \widetilde{U}N e_r e_r^* U_\#^*  + U_\# e_r e_r^* N^* \widetilde{U}^* U_{2} D_{2} {U_{2}}^{\top} \right).
    \end{align*}
    Note that without loss of generality, we have let $U_2$ be the first column of $\widetilde{U}$. Thus we have
    \begin{align*}
        D\phi (Z_\#)[\xi] &= \xi + \alpha \cdot \Big( U_{2} D_{2} (1,0,\ldots, 0) (N e_r)
        \begin{pmatrix}
            {U_{2}}^{\top} \\
            0 \\
            \vdots \\
            0
        \end{pmatrix} 
        + (U_{2},\,0,\ldots, 0) (e_r^* N^*) 
        \begin{pmatrix}
            1 \\
            0 \\
            \vdots \\
            0
        \end{pmatrix} 
        D_{2} {U_{2}}^{\top} \Big) \\
        &=\xi + \alpha \cdot 2N(1,1) \cdot U_{2} D_{2} {U_{2}}^{\top}.
    \end{align*}
    We can immediately read the eigenvalues and eigenvectors of $D\phi (Z_\#)$ from the above expression. Specifically, when $\F=\R$, $D\phi (Z_\#)$ has
    \begin{enumerate}[label=(\arabic*)]
        \item One eigenvector
        $\xi = U N \widetilde{U}^* + \widetilde{U}N U^* $ with $N = \begin{pmatrix}
            1 & 0 &\ldots &0 \\
            0 & 0 &\ldots &0 \\
            \vdots & \vdots & &  \vdots \\
            0 & 0 &\ldots &0
        \end{pmatrix}$, whose corresponding eigenvalue is $\lambda = 1+2\alpha\cdot D_{2} >1$;
        \item $(\frac{r(2n-r+1)}{2}-1)$ eigenvectors with eigenvalues $\lambda = 1$. 
    \end{enumerate}
    The case $\F=\C$ is similar except that the dimensionality is different.  
    
    Now that $D\phi(Z_\#)$ has one eigenvalue greater than 1, while the rest of the eigenvalues are equal to 1. By {\cite[Theorem III.7]{shub2013global}}, there is an unstable manifold and a center manifold in the neighborhood of $Z_\#$, which can be extended globally. The existence of the unstable manifold ensures that $Z_\#$ is an asymptotic unstable fixed point of the iteration function $\phi(Z_\#)$. Thus, the Riemannian gradient descent algorithm with varying step size (\ref{eq:vary}) almost surely escapes $\mathcal{S}_{r-1}$.
    
    In particular, $D\phi(Z_\#)$ is always a local diffeomorphism independent of the choice of $\alpha$, as its only eigenvalues are $1$ and $1+\alpha D_{2}$. Therefore, the result of Lemma \ref{lemma:vary} holds true for arbitrarily large $\alpha>0$.
\end{proof}

\section{Numerical experiments}
\label{sec:numerical}

In this section, we present some numerical experiments to illustrate our theoretical results in Theorem \ref{thm:GF} and \ref{lemma:vary}. We also provide some evidence in support of conjectures beyond the previous theorem and lemma. 

In all experiments, we let $\F = \R$, $m = n =100$, $r = 5$, and we use the same ground truth matrix $X \in \M_r$ with distinct singular values. We use the Riemannian gradient descent algorithm to minimize $f(Z) = \frac{1}{2}\|Z-X\|_F^2$. The experiments only differ by the sampling rule and the choice of the step sizes $\alpha_k$. Each figure is generated by repeating the experiment 100 times. The shaded area represents the range of the data and the solid line represents the median.  

\begin{figure}[ht]
    \centering
    \begin{subfigure}[t]{.4\linewidth}
        \includegraphics[width = \linewidth, align = c]{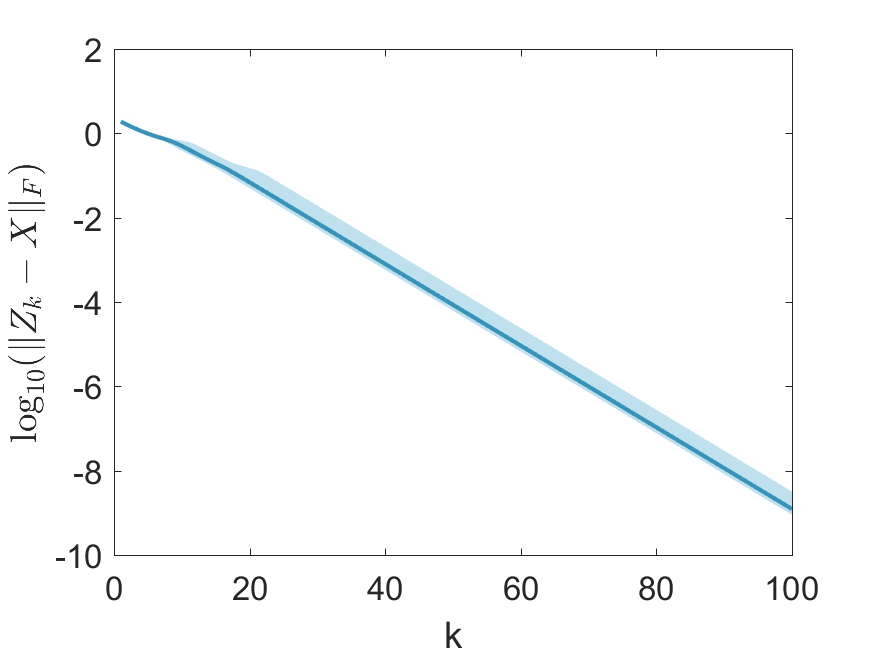}
        \caption{Local escape near $\mathcal{S}_{r-1}$}
        \label{fig:rank_r_1-dist}
    \end{subfigure}
    \hspace{0.05\textwidth}
    \begin{subfigure}[t]{.4\linewidth}
        \includegraphics[width = \linewidth, align = c]{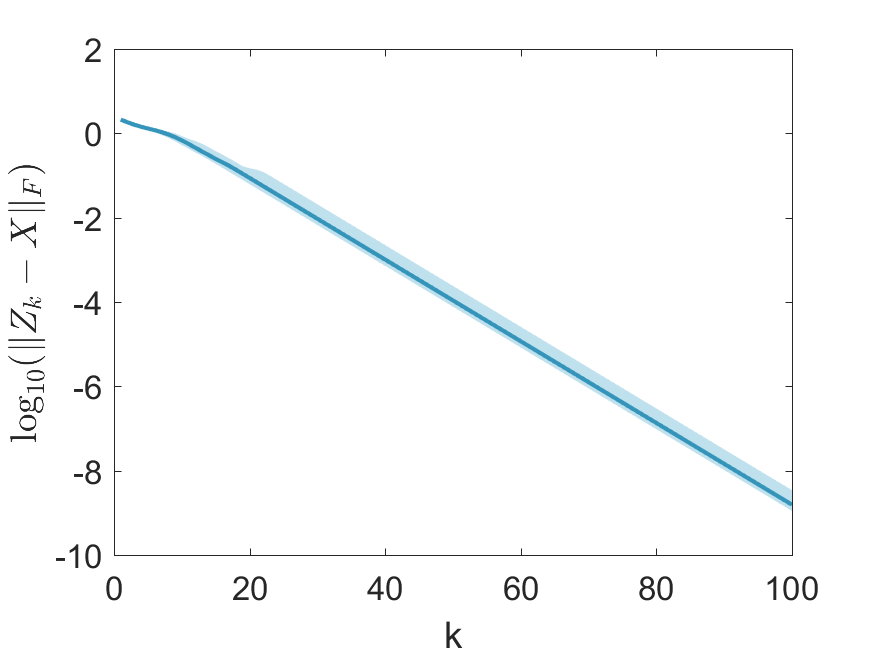}
        \caption{Local escape near $\mathcal{S}_{r-2}$}
        \label{fig:rank_r_2-dist}
    \end{subfigure}
    \caption{Escape of spurious critical points}

\end{figure}

The first experiment is performed near a rank-$(r-1)$ spurious critical point $Z_\#^{(1)} \in \mathcal{S}_{r-1}$. The initial points are randomly sampled in the local neighborhood of $Z_\#^{(1)}$. The stepsize is fixed to be $\alpha_k \equiv \alpha=0.2$. Figure \ref{fig:rank_r_1-dist} shows the log10 distance between $Z_k$ and $X$. It can be seen that in all the repeated experiments, the sequence always succeeds to escape $Z_{\#}^{(1)}$ and converge to $X$. 

To verify whether $\mathcal{S}_{s}$ ($s<r-1$) incurs the same behavior, we repeat the experiment with $Z_\#^{(2)} \in \mathcal{S}_{r-2}$. It can be seen in Figure \ref{fig:rank_r_2-dist} that the phenomenon is indeed the same. Thus we conjecture that a similar result as Theorem \ref{thm:GF} holds for those $\mathcal{S}_{s}$ with $s<r-1$ as well. Proof of such result is left for future work.

\begin{figure}[ht]
    \centering
    \begin{subfigure}[t]{.4\linewidth}
        \includegraphics[width = \linewidth, align = c]{figures/rank_r_1-dist.png}
        \caption{$\log_{10}(\|Z_k-X\|_F)$, fixed stepsize}
        \label{fig:global-dist}
    \end{subfigure}
    \hspace{0.05\textwidth}
    \begin{subfigure}[t]{.4\linewidth}
        \includegraphics[width = \linewidth, align = c]{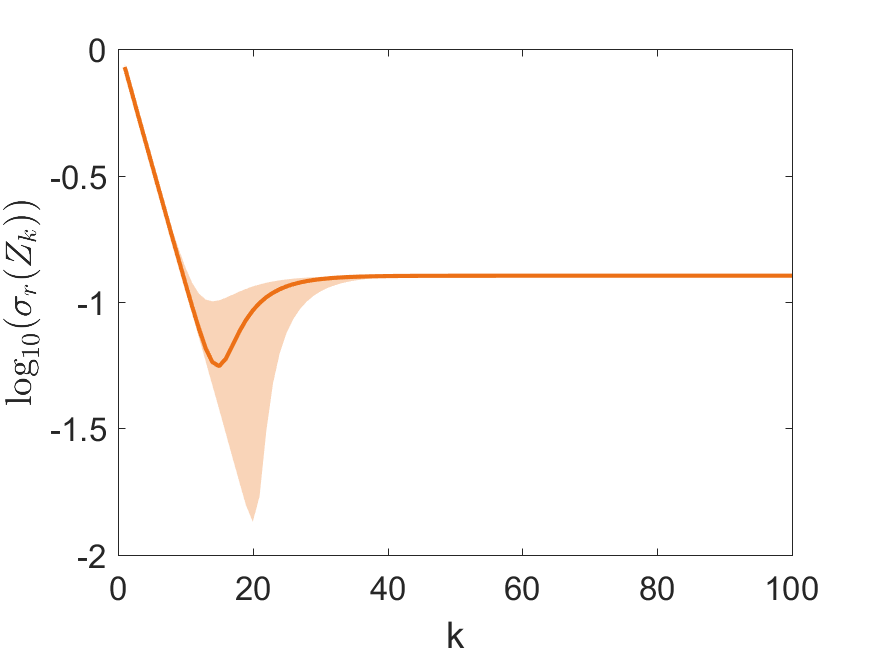}
        \caption{$\log_{10}(\sigma_r(Z_k))$, fixed stepsize}
        \label{fig:global-sigmin}
    \end{subfigure}

    \begin{subfigure}[t]{.4\linewidth}
        \includegraphics[width = \linewidth, align = c]{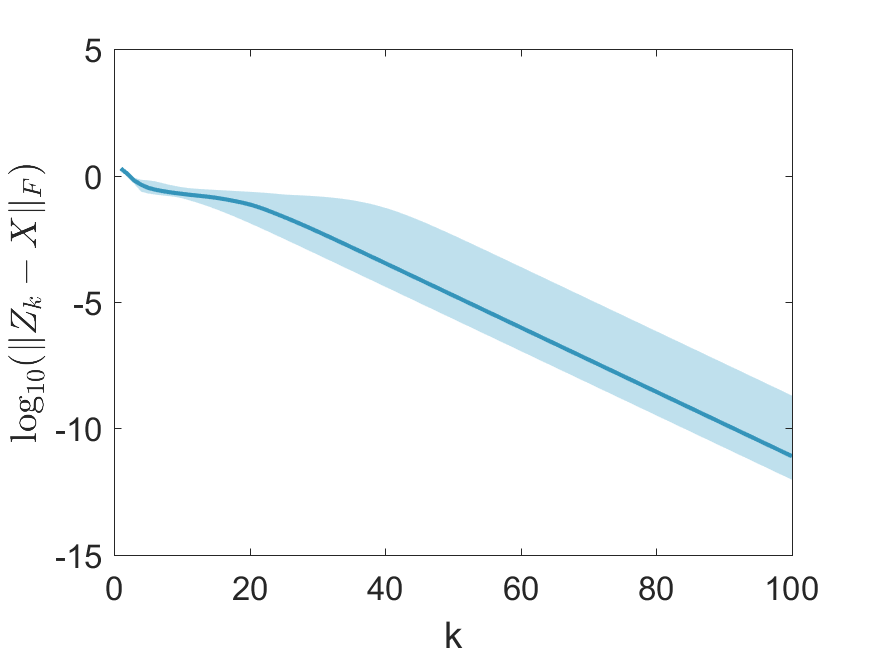}
        \caption{$\log_{10}(\|Z_k-X\|_F)$, varying stepsize}
        \label{fig:varying-dist}
    \end{subfigure}
    \hspace{0.05\textwidth}
    \begin{subfigure}[t]{.4\linewidth}
        \includegraphics[width = \linewidth, align = c]{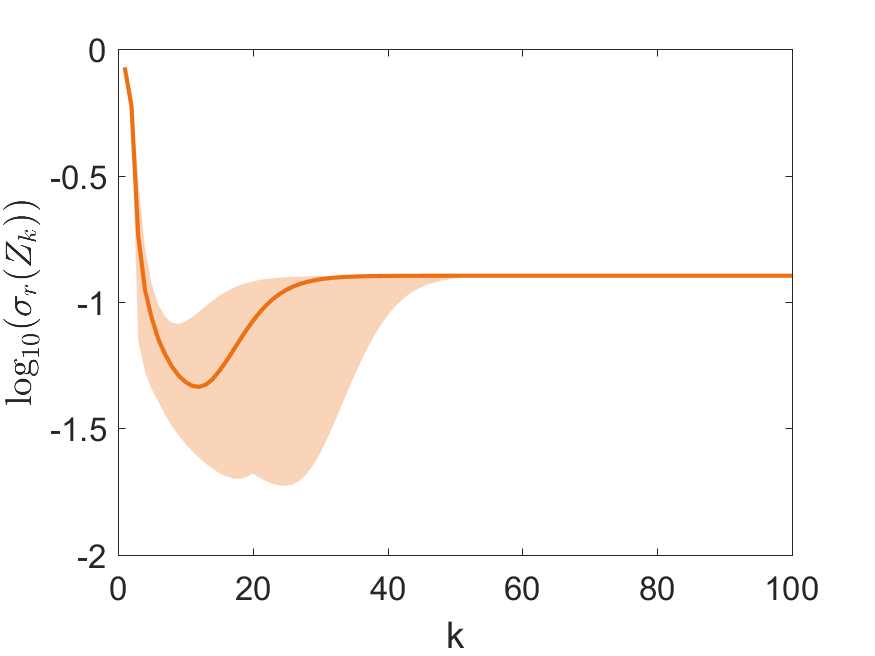}
        \caption{$\log_{10}(\sigma_r(Z_k))$, varying stepsize}
        \label{fig:varying-sigmin}
    \end{subfigure}
    \caption{Comparison of fixed and varying stepsizes}
    \label{fig:global_vary}
\end{figure}

Next, we investigate Lemma \ref{lemma:vary} and the varying step size $\alpha_k = \alpha \cdot \sigma_r(Z_k)$. Figures \ref{fig:global-dist} and \ref{fig:global-sigmin} are the results with a fixed stepsize $\alpha_k \equiv 0.2$. Figures \ref{fig:varying-dist} and \ref{fig:varying-sigmin} are the results with varying stepsizes $\alpha_k = 2 \sigma_r(Z_k)$. The left are the distances to the ground truth $X$. The right are the log values of $\sigma_r(Z_k)$ along the iterative path. We can see that first of all, the iterative sequences always escape all spurious critical points and converge to the ground truth. Moreover, the value of $\sigma_r(Z_k)$ is never too small, but soon converges to the smallest singular value of $X$. This helps illustrate that the varying stepsize $\alpha_k = \alpha \cdot \sigma_r(Z_k)$ is \emph{not a diminishing stepsize} in practice, but is rather always above a certain value.

\section{Discussion}
\label{sec:discussion}
In this paper, we discuss the asymptotic escape of the spurious critical points on the low-rank matrix manifold. The goal is to shed some light on the incompleteness of the low-rank matrix manifold $\M_r$ and justify the global use of Riemannian gradient descent on the manifold. To this end, we first point out the existence of a set of spurious critical points $\mathcal{S}_{\#} \subset \overline{\M_r}\backslash\M_r$ and discuss its singularity. We then use a rescaled gradient flow combined with the dynamical low-rank approximation to describe the local landscape, which enables us to eliminate the singularity and prove the asymptotic escape result. We also present a corresponding result for the gradient descent. Numerical experiments are provided to illustrate the theoretical results. 

Though this study is focused on $\mathcal{S}_{r-1}$, the asymptotic escape is empirically observed for $\mathcal{S}_s$ with $s\le r-2$ as well. In fact, all spurious critical points in $\mathcal{S}_{\#}$ are observed to be asymptotically unstable in practice, which can be seen from the numerical experiments. The current rescaled gradient flow (\ref{DLRA*}) loses both $C^0$- and $C^1$-extensions at $\mathcal{S}_s$ with $s\le r-2$. This is because the continuity of eigenvalues and eigenvectors are only possible when only one of the eigenvalues is approaching zero. Extension of the result to the case $s\le r-2$ is left for future work. On the other hand, the assumption that the eigenvalues of $X$ are distinct is not an essential assumption, and can easily be removed.

Even though the result for the gradient descent calls for a step size $\alpha\cdot \sigma_{r}(Z)$, there is no upper bound on the constant $\alpha$ from the asymptotic escape analysis. It is because the isomorphism requirement will not be violated even for arbitrarily large $\alpha$. Thus the step size criterion is not more stringent than that in classical saddle escape results, where there is usually an upper bound on the step size.

In addition to the asymptotic result in this paper, a non-asymptotic result on the number of steps needed to escape the spurious critical points can be found in \cite{hou2020fast}. There it is shown that the converging set of the spurious critical points can be upper bounded by a small positive measure. With high probability, one has nearly linear convergence rate towards the ground truth. The two sides of the story complement each other and provide a wholesome picture of the unique structure of the low-rank matrix manifold.

\bibliographystyle{plain}
\bibliography{reference}

\appendix
\section{Auxiliary lemmas}

\begin{lemma}[The $\text{sin}\Theta$ Theorem,  \cite{davis1970rotation}]
\label{lemma:sin_theta}
    Let $A$ be a Hermitian operator. Assume that
    \begin{align*}
        A = 
        \begin{pmatrix}
            E_0 & E_1
        \end{pmatrix}
        \begin{pmatrix}
            A_0 & 0 \\
            0 & A_1
        \end{pmatrix}
        \begin{pmatrix}
            E_0^* \\
            E_1^*
        \end{pmatrix}
    \end{align*}
    is an invariant subspace decomposition (i.e., a generalized eigenvalue decomposition) of $A$. Let
    \begin{align*}
        B = A + \Delta, \quad B = 
        \begin{pmatrix}
            F_0 & F_1
        \end{pmatrix}
        \begin{pmatrix}
            B_0 & 0 \\
            0 & B_1
        \end{pmatrix}
        \begin{pmatrix}
            F_0^* \\
            F_1^*
        \end{pmatrix}.
    \end{align*}
    Let $\Theta_0$ be the angle matrix between subspaces $E_0$ and $F_0$. Define the residual as 
    \begin{align*}
        R := BE_0 - E_0A_0.
    \end{align*}
    If there is an interval $[\beta, \alpha]$ and $\delta >0$, such that the spectrum of $A_0$ lies entirely in $[\beta, \alpha]$, while that of $B_1$ lies entirely in $(-\infty, \beta-\delta] \cup [\alpha+\delta, +\infty)$, then for every unitary-invariant norm $\|\cdot\|$, we have
    \begin{align*}
        \delta \|\text{sin}\Theta_0\| \le \|R\|.
    \end{align*}
    In particular, this holds true for the matrix 2-norm and the Frobenius norm. 
\end{lemma}


\section{Proof of Lemma \ref{lemma:submanifold}}
\label{sec:proof_submanifold}
We recall the lemma from the main text.
\begin{manuallemma}{3.2} Define 
    \begin{align*}
        \N_{Z_\#}:= \left\{(U_\#, \, S_\#): \,\, U_\# = (U_{1}, U_{3})P_\#^*, \,\, S_\# = P_\#
        \begin{pmatrix}
            D_{1} & 0 \\ 0 & 0
        \end{pmatrix}
        P_\#^*, \,\, U_3 \perp U_x
        \right\}.
    \end{align*}
    Then $\N_{Z_\#}$ is an embedded submanifold of the manifold $\M := \St(n,r) \oplus \S_r$.
\end{manuallemma}

\blue{The intuition behind Lemma \ref{lemma:submanifold} is that the set $\N_{Z_\#}$ is a subset of $\M$ characterized by some algebraic constraints, namely $U_\#S_\#U_\#^* = Z_\#$ and $U_3 \perp U_x$. As is often the case, one would expect such algebraic constraints to give an embedded submanifold. We will make this intuition rigorous in this section.

We note that traditionally, embedded submanifold is proved by the submersion theorem, i.e., by showing that the set is the preimage of a regular value of a submersive mapping. But this approach does not work here because $Z_\#$ is not a regular value. Instead, we need to go back to the definition of a submanifold and construct chart functions on $\N_{Z_\#}$ directly.
}


Below are some auxiliary results from the literature. 

\begin{lemma}[{\cite[Proposition 3.3.2]{absil2009optimization}}]
    \label{lemma:submanifold_property}
    A subset $\N$ of a manifold $\M$ is a $d$-dimensional embedded submanifold of $\M$ if and only if, around each point $x\in\N$, there exists a chart $(\U,\varphi)$ of $\M$ such that $\N\cap\U$ is a $\varphi$-coordinate slice of $\U$, i.e.,
    \begin{align*}
        \N \cap \U = \{x\in\U:\,\, \varphi(x) \in \R^{d} \times \mathbf{0}\}.
    \end{align*}
    In this case, the chart $(\N \cap \U, \varphi)$, where $ \varphi$ is seen as a mapping into $\R^d$, is a chart of the embedded submanifold $\N$.
\end{lemma}
By Lemma \ref{lemma:submanifold_property}, if we can construct an atlas of $\M$ and an atlas of $\N_{Z_\#}$, such that the charts in the latter atlas are coordinate slices of the charts in the former atlas, then $\N_{Z_\#}$ is an embedded submanifold of $\M$. This approach is less common than the traditional submersion theorem approach, but is necessary for our problem.

\begin{lemma}[\cite{fraikin2007optimization}]
    \label{lemma:stiefel}
    For the real Stiefel manifold $\St(n,k)$, there exists an atlas $\cup_{Q}(\mathcal{U}_Q, \varphi_Q)$ of the Stiefel manifold. Namely, for each chart $(\mathcal{U}_Q, \varphi_Q)$, $Q$ is a matrix in $\St(n,k)$, and the function $\varphi_Q$ can be expressed as
    \begin{align*}
        \varphi_Q: \quad \mathcal{U}_Q &\to \text{Skew}(k) \oplus \R^{(n-k)\times k}, \\
        U &\mapsto (\Omega_{11}, \Omega_{21}),
    \end{align*}
    where
    \begin{align*}
        \Omega_{11} &= (U_1^\top + Q_1^\top)^{-1}\left(Q_1^\top U_1 + U_2^\top Q_2 - U_1^\top Q_1- Q_2^\top U_2\right)(U_1+Q_1)^{-1}, \quad \Omega_{11} = - \Omega_{11}^\top,\\
        \Omega_{21} &= (U_2 - Q_2) (U_1+Q_1)^{-1},
    \end{align*}
    and $U = \begin{pmatrix}
        U_1 \\
        U_2
    \end{pmatrix}$, 
    $Q = \begin{pmatrix}
        Q_1 \\
        Q_2
    \end{pmatrix}$ are the block forms of $U$ and $Q$ respectively. Such chart function is defined on the subset $\mathcal{U}_Q \subset \St(n,k) $ which covers all of the manifold $\St(n,k)$ except a zero-measure set. 
    
    In particular, if $Q = \begin{pmatrix}
        I_k \\
        \mathbf{0}
    \end{pmatrix}$, then 
    \begin{align*}
        \Omega_{11} &= (U_1^\top + I_k)^{-1}\left(U_1 - U_1^\top \right)(U_1+I_k)^{-1},\\
        \Omega_{21} &= U_2 (U_1+I_k)^{-1}.
    \end{align*}
\end{lemma}
Lemma \ref{lemma:stiefel} provides a neat construction of charts on the Stiefel manifold. In fact, we only need two charts to cover the whole manifold, if we choose any two $Q$'s that do not share any left singular vector. We will use this construction frequently in the proof of Lemma \ref{lemma:submanifold}.

We are now ready to prove Lemma \ref{lemma:submanifold}.
\begin{proof}[\textbf{Proof of Lemma \ref{lemma:submanifold}}]
    We restrict our attention to the case $\F=\R$. The case $\F = \C$ is very similar except that the dimensionalities of some manifolds in the subsequent proof are slightly different.
    
    We aim to construct explicit charts of $\M = \St(n,r) \oplus \S_r$, and explicit charts of $\N_{Z_\#}$, such that the latter are the coordinate slices of the former. For clarity, we will first write out the charts of $\N_{Z_\#}$, and then express them as coordinate slices of charts of $\M$. 
    
    \smallskip
    \noindent \textbf{Step 1: Construct charts of $\N_{Z_\#}$.}
    
    For any $(U,S) \in \N_{Z_\#}$, we rewrite $U$ and $S$ as the following: 
    \begin{align*}
        &U = U_1 P_1^\top + U_3 P_2^\top, \qquad S = P_1 D_1 P_1^\top, \\
        & \text{where} \quad P_1 \in \R^{r \times s}, \quad P_2 \in \R^{r \times (r-s)}, \quad P = (P_1, P_2) \in \text{SO}(r).
    \end{align*}
    We argue that there exists a mapping from every $P_1$ to a unique $P_2$. An intuitive explanation is that $P_2$ can always be uniquely determined by a Gram-Schmidt process starting from the identity matrix. Thus we can write $P_2 = \mathcal{P}_2(P_1)$ where $\mathcal{P}_2: \R^{r \times s} \to \R^{r \times (r-s)}$ is a function.
    Therefore, any $(U,S) \in \N_{Z_\#}$ can be re-parameterized using only $(P_1, U_3)$. We write this re-parameterization as a function $f$:
    \begin{align*}
        f:\quad  
        \N_{Z_\#} &\to \St(r,s) \oplus \widetilde{\St}(n,r-s; \, U_x); \\
        (U,S) &\mapsto  (P_1, U_3).
    \end{align*}
    Here $\St(r,s)$ is a Stiefel manifold, and $\widetilde{\St}(n,r-s; \, U_x)$ is a constrained Stiefel manifold:
    \begin{align*}
        \widetilde{\St}(n,r-s; \, U_x) := \left\{U_3: \,\, U_3 \in \St(n,r-s), \,\, U_3 \perp U_x \right\},\quad \text{where }U_x = (U_1, U_2).
    \end{align*}
    
    We now construct charts for $P_1$ and $U_3$ respectively. The domain of $P_1$ is the Stiefel manifold $\St(r,s)$. By Lemma \ref{lemma:stiefel}, there exists an atlas where every chart function maps to $\text{Skew}(s) \oplus \R^{(r-s)\times s}$. Let $g^{(1)}$ be one such chart function:
    \begin{align*}
        g^{(1)}: \quad \St(r,s) &\to \text{Skew}(s) \oplus \R^{(r-s)\times s} \\
         P_1 &\mapsto (\Omega_{11}, \Omega_{21}).
    \end{align*}
    
    The domain of $U_3$ is the constrained Stiefel manifold $\widetilde{\St}(n,r-s; \, U_x)$. Here $U_x \in \St(n,r)$ is the eigenvectors matrix of the ground truth $X$, which is fixed. To construct a chart of $\widetilde{\St}(n,r-s; \, U_x)$, we first construct a mapping $g^{(2)}$ according to Lemma \ref{lemma:stiefel}, such that 
    \begin{align*}
        g^{(2)}: \quad \widetilde{\St}(n,r-s; \, U_x) &\to \text{Skew}(r-s) \oplus \widetilde{\R}^{(n-(r-s))\times (r-s)}, \\
        U_3 &\mapsto (\Lambda_{11},\Lambda_{21}).
    \end{align*}
    The domain of $\Lambda_{21}$ is $\widetilde{\R}^{(n-(r-s))\times (r-s)}$, which is a constrained set. To express the constraints $U_3 \perp U_x$ in terms of constraints on $\Lambda_{21}$, we write 
    $U_3 = \begin{pmatrix}
        U_{3,1} \\
        U_{3,2}
    \end{pmatrix}$ and 
    $U_x = \begin{pmatrix}
        U_{x,1} \\
        U_{x,2}
    \end{pmatrix}$. Assume without loss of generality that $g^{(2)}$ is constructed by picking $Q = (I_{r-s}, \mathbf{0})^\top$ in Lemma \ref{lemma:stiefel}. Then
    \begin{align*}
        \Lambda_{21} = U_{3,2} (U_{3,1} + I_{r-s})^{-1}.
    \end{align*}
    Since $U_3 \perp U_x$, we have 
    \begin{align*}
        U_x^\top U_3 = U_{x,1}^\top U_{3,1} + U_{x,2}^\top U_{3,2} = 0.
    \end{align*}
    Thus,
    \begin{align*}
        U_{x,2}^\top U_{3,2} = -U_{x,1}^\top U_{3,1}.
    \end{align*}
    This gives us
    \begin{align*}
        U_{x,2}^\top \Lambda_{21} = -U_{x,1}^\top  U_{3,1}(U_{3,1} + I_{r-s})^{-1}.
    \end{align*}
    These are linear constraints on $\Lambda_{21}$.
    
    Let $g$ be the concatenation of $g^{(1)}$ and $g^{(2)}$, then we have a re-parameterization of $(P_1, U_3)$ as follows:
    \begin{align*}
        g: \quad \St(r,s) \oplus \widetilde{\St}(n,r-s) & \to \text{Skew}(s) \oplus \R^{(r-s)\times s} \oplus \text{Skew}(r-s)\oplus  \widetilde{\R}^{(n-(r-s))\times (r-s)};\\
        (P_1, U_3) &\mapsto (\Omega_{11}, \Omega_{21}, \Lambda_{11}, \Lambda_{21}).
    \end{align*}
    Here $\widetilde{\R}^{(n-(r-s))\times (r-s)}$ is the submanifold of $\R^{(n-(r-s)) \times (r-s)}$ defined by the linear constraints that we derived:
    \begin{align*}
        \widetilde{\R}^{(n-(r-s)) \times (r-s)} := \left\{\Lambda_{21} \in \R^{(n-(r-s)) \times (r-s)} : \,\, U_{x,2}^\top \Lambda_{21} = -U_{x,1}^\top  U_{3,1}(U_{3,1} + I_{r-s})^{-1}\right\}.
    \end{align*}

    Let $\Lambda_{21}^\circ$ be an arbitrary solution to the equation $U_{x,2}^\top \Lambda_{21} = -U_{x,1}^\top  U_{3,1}(U_{3,1} + I_{r-s})^{-1}$. Then 
    \begin{align*}
        \widetilde{\R}^{(n-(r-s)) \times (r-s)} = \Lambda_{21}^\circ + \text{Ker}(U_{x,2}^\top).
    \end{align*}
    By finding an orthogonal basis for $\text{Ker}(U_{x,2}^\top)$, it is easy to construct a chart function 
    \begin{align*}
        h: \quad \widetilde{\R}^{(n-(r-s))\times (r-s)} &\to \R^{(n-(r-s))(r-s) - r(r-s)} \\
        \Lambda_{21} &\mapsto \Gamma.
    \end{align*} 
    Putting everything together, we have that 
    \begin{align*}
        \varphi:= (\text{id}, h) \circ g \circ f: \quad \N_{Z_\#} &\to \text{Skew}(s) \oplus \R^{(r-s) \times s} \oplus \text{Skew}(r-s)\oplus \R^{(n-(r-s))(r-s) - r(r-s)}; \\
        (U,S) &\mapsto (\Omega_{11}, \Omega_{21}, \Lambda_{11}, \Gamma).
    \end{align*}
    This is a chart function for the whole $\N_{Z_\#}$ except a zero-measure set. Varying $g^{(1)}$ and $g^{(2)}$ as needed and we have the atlas for the whole $\N_{Z_\#}$.
    
    \smallskip
    \noindent \textbf{Step 2: Express the charts of $\N_{Z_\#}$ as coordinate slices of charts of $\M$.}
    
    To express things into coordinate slices, we will work the other way around: we extend the chart function $\varphi$ into a chart function $\widetilde{\varphi}$ defined on $\M = \St(n,r) \oplus \S_r$. 
    
    For any $(U,S)\in\M = \St(n,r) \oplus \S_r$, we construct a re-parameterization as follows:  
    \begin{align*}
        &U = \Big( (U_1 R_1, 0)  + (M_4, U_3R_2) \Big)
        \begin{pmatrix}
            P_1^\top \\
            P_2^\top
        \end{pmatrix}, \qquad 
        S = (P_1, P_2) \widetilde{S}
        \begin{pmatrix}
            P_1^\top \\
            P_2^\top
        \end{pmatrix}, \\
        &\text{where} \quad P_1 \in \R^{r \times s}, \quad P_2 \in \R^{r \times (r-s)}, \quad P = (P_1, P_2) \in \text{SO}(r), \\
        &  U_3 \in \widetilde{\St}(n,r-s; \, U_1), \quad M_4\in \widetilde{\R}^{n\times s}, \quad \widetilde{S} \in \S_r,\\
        & R_1 \in \text{upper}(s,s), \quad R_2\in\widetilde{\text{upper}}(r-s,r-s).
    \end{align*}
    The domain of $P_1$ is $\text{St}(r,s)$. $P_2$ is still uniquely determined by $P_1$ as before. 
    The domain of $U_3$ is the constrained Stiefel manifold $\widetilde{\St}(n,r-s; \, U_1):= \left\{U_3: \,\, U_3 \in \St(n,r-s), \,\, U_3 \perp U_1 \right\}$. Note that the constraints are only in terms of $U_1$ instead of $U_x = (U_1, U_2)$. 
    The domain of $M_4$ is the linearly constrained subspace $\widetilde{\R}^{n\times s} := \left\{M_4 \in  \R^{n\times s}, \, M_4 \perp U_1\right\}$. The domain of $\widetilde{S}$ is $\S_r$. The domain of $R_1$ is the subspace of $s\times s$ upper triangular matrices. The domain of $R_2$ is the subspace of $(r-s) \times (r-s)$ upper triangular matrices, but with some constraints that will be specified later. We define the following mapping: 
    \begin{align*}
        \widetilde{f}:  \quad \St(n,r) \oplus \S_r &\to \St(r,s) \oplus \widetilde{\St}(n,r-s; \, U_1)  \oplus \widetilde{\R}(n,s)  \oplus \widetilde{\text{upper}}(r-s,r-s) \oplus  \S_r \oplus \text{upper}(s,s);\\
        (U,S) &\mapsto \left(P_1, U_3, M_4, R_2-I_{r-s},
        \widetilde{S} -
        \begin{pmatrix}
            D_1 & 0 \\
            0 & 0
        \end{pmatrix},
        R_1-I_s\right). 
    \end{align*}
    The mapping $\widetilde{f}$ is written in such a way because, if $(U,S) \in \N_{Z_\#}$, then the last few components are all zero, and $f$ is just a coordinate slice of $\widetilde{f}$:
    \begin{align*}
        \widetilde{f}(U,S) = \left(P_1, U_3, 0, 0,
        0, 0\right).
    \end{align*}
    
    For the first part of the image of $\widetilde{f}$, we apply $g$ as before:
    \begin{align*}
        g: \quad \St(r,s) \oplus \widetilde{\St}(n,r-s) & \to \text{Skew}(s) \oplus \R^{(r-s)\times s} \oplus \text{Skew}(r-s)\oplus  \widehat{\R}^{(n-(r-s))\times (r-s)};\\
        (P_1, U_3) &\mapsto (\Omega_{11}, \Omega_{21}, \Lambda_{11}, \Lambda_{21}).
    \end{align*}
    However, the set $\widehat{\R}^{(n-(r-s))\times (r-s)}$ is different from the $\widetilde{\R}^{(n-(r-s))\times (r-s)}$ before, because the constraints only contain $U_1$ but does not contain $U_2$. Fewer constraints mean a larger subspace, and we have
    \begin{align*}
        \widehat{\R}^{(n-(r-s))\times (r-s)} &= \left\{\Lambda_{21} \in \R^{(n-(r-s)) \times (r-s)} : \,\, U_{1,2}^\top \Lambda_{21} = -U_{1,1}^\top  U_{3,1}(U_{3,1} + I_{r-s})^{-1}\right\} \\
        &= \widetilde{\R}^{(n-(r-s))\times (r-s)} + \left(\text{Ker}(U_{1,2}^\top) \backslash \text{Ker}(U_{x,2}^\top)\right) \\
        &=\Lambda_{21}^\circ + \text{Ker}(U_{x,2}^\top) + \left(\text{Ker}(U_{1,2}^\top) \backslash \text{Ker}(U_{x,2}^\top)\right).
    \end{align*}
    Let $h^{(2)}$ be the chart function for the extra subspace, then 
    \begin{align*}
        (h, h^{(2)}): \quad \widehat{\R}^{(n-(r-s))\times (r-s)} &\to  \R^{(n-(r-s))(r-s) - r(r-s)} \oplus \R^{(r-s)(r-s)},\\
        \Lambda_{21} &\mapsto (\Gamma, \Gamma^{(2)}).
    \end{align*}
    Putting them together, we have
    \begin{align*}
         (\text{id}, h, h^{(2)}) \circ g \circ f: \quad (P_1, U_3) \mapsto (\Omega_{11}, \Omega_{21}, \Lambda_{11}, \Gamma, \Gamma^{(2)}).
    \end{align*}
    The chart function $\varphi$ is a coordinate slice of the above mapping.
    
    It suffices to find the chart functions for the remaining components of $\widetilde{f}(U,S)$, i.e., the components $M_4$, $\tilde{S}$, $R_1$, $R_2$.
    For $\tilde{S} \in \S_r$ and $R_1 \in \text{upper}(s,s)$, the domains are Euclidean spaces with natural bases. We now look at $M_4$ and $R_2$. 
    
    Decompose $M_4$ into parts that are parallel to and perpendicular to the subspace of $U_3$: 
    \begin{align*}
        M_4 = M_4^\parallel + M_4^\perp, \quad \text{where } M_4^\parallel = P_{U_3}M_4, \quad M_4^\perp = P_{U_3}^\perp M_4.
    \end{align*}
    Let $M_4^\perp = U_4 R_4$ be the QR decomposition of $M_4^\perp$. Then the whole $M_4$ could be written as 
    \begin{align*}
        M_4 = (U_3, U_4) 
        \begin{pmatrix}
            R_3 \\
            R_4
        \end{pmatrix}, \qquad R_3 \in \R^{(r-s) \times (r-s)}, \quad R_4 \in \text{ upper}(s,s).
    \end{align*}
    In this way, we can re-parameterize $(M_4, R_2)$ using $(U_4, R_2, R_3, R_4)$:
    \begin{align*}
        p: 
        (M_4, R_2-I_{r-s}) &\mapsto (U_4, R_2, R_3, R_4).
    \end{align*}
    The domain of $U_4$ is the constrained Stiefel manifold $\widetilde{\St}(n,s; U_1, U_3)$. Just as before, we can construct a composite function for this constrained Stiefel manifold:
    \begin{align*}
        g^{(3)}: \quad \widetilde{\St}(n,s; U_1, U_3) &\to \text{Skew}(s) \oplus \widetilde{\R}^{(n-s) \times s}, \\
        U_4 &\mapsto (\Pi_{11}, \Pi_{21}); \\
        h^{(3)}: \quad \widetilde{\R}^{(n-s) \times s} &\to \R^{(n-s)s - rs}, \\
        \Pi_{21} &\mapsto \Xi; \\
        (\text{id}, h^{(3)}) \circ g^{(3)}: \quad \widetilde{\St}(n,s; U_1, U_3) &\to \text{Skew}(s) \oplus \R^{(n-s)s - rs}, \\
        U_4 &\mapsto (\Pi_{11}, \Xi).
    \end{align*}
    The remaining components are $R_2$, $R_3$, and $R_4$. The constraints for them come from the requirement that $U$ as a whole is in $\St(n,r)$. This gives
    \begin{align*}
        U^\top U &= \Big( (U_1 R_1, 0)  + (M_4, U_3R_2) \Big)^\top \Big( (U_1 R_1, 0)  + (M_4, U_3R_2) \Big) \\
        &= \begin{pmatrix}
            R_1^\top U_1^\top U_1 R_1 & \mathbf{0} \\
            \mathbf{0} & \mathbf{0}
        \end{pmatrix} 
        + \left( (U_4, U_3)
        \begin{pmatrix}
            R_4 & 0 \\
            R_3 & R_2
        \end{pmatrix} \right)^\top
        \left( (U_4, U_3)
        \begin{pmatrix}
            R_4 & 0 \\
            R_3 & R_2
        \end{pmatrix} \right) \\
        &= \begin{pmatrix}
            R_1^\top R_1 & \mathbf{0} \\
            \mathbf{0} & \mathbf{0}
        \end{pmatrix}
        + \begin{pmatrix}
            R_4 & 0 \\
            R_3 & R_2
        \end{pmatrix}^\top
        \begin{pmatrix}
            R_4 & 0 \\
            R_3 & R_2
        \end{pmatrix} \\
        &= \begin{pmatrix}
            R_1^\top R_1+R_3^\top R_3 + R_4^\top R_4 & R_3^\top R_2 \\
            R_2^\top R_3 & R_2^\top R_2
        \end{pmatrix} = I_r
    \end{align*}
    Denote 
    \begin{align*}
        R_0 := \begin{pmatrix}
            R_2 & R_3 \\
            \mathbf{0} & R_4
        \end{pmatrix} \in \text{ upper}(r,r).
    \end{align*}
    Then the $r \times r$ upper-triangular matrix $R_0$ should satisfy
    \begin{align*}
        R_0^\top R_0 =  
        \begin{pmatrix}
            I_{r-s} & \mathbf{0} \\
            \mathbf{0} & I_s - R_1^\top R_1 
        \end{pmatrix}.
    \end{align*}
    Such $R_0$ is uniquely determined. 
    Therefore, we get the following chart function for the components $(M_4, R_2)$:
    \begin{align*}
        (\text{id}, h^{(3)}) \circ g^{(3)}\circ p:  \quad 
        \widetilde{\R}(n,s) \oplus  \widetilde{\text{upper}}(r-s,r-s) &\to \text{Skew}(s) \oplus \R^{(n-s)s - rs}, \\
        (M_4, R_2-I_{r-s}) &\mapsto (\Pi_{11}, \Xi).
    \end{align*}

    Putting everything together, we get the following chart function for the manifold $\M$:
    \begin{align*}
        \widetilde{\varphi}:= \Big((\text{id}, h, h^{(2)}) \circ g, &\,\, (\text{id}, h^{(3)}) \circ g^{(3)}\circ p, \,\,\text{id}\Big) \circ \widetilde{f}: \\
        \M &\to \Big(\text{Skew}(s) \oplus \R^{(r-s)s} \oplus \text{Skew}(r-s)\oplus \R^{(n-2r+s)(r-s)} \oplus \R^{(r-s)(r-s)} \Big) \\
        & \qquad \qquad \qquad \oplus \Big( \text{Skew}(s) \oplus \R^{(n-s-r)s} \Big)
        \oplus \S_r \oplus \text{upper}(s,s), \\
        (U,S) &\mapsto \left( \Big(\Omega_{11}, \Omega_{21}, \Lambda_{11}, \Gamma, \Gamma^{(2)}\Big), \Big(\Pi_{11}, \Xi\Big), \widetilde{S} -
            \begin{pmatrix}
                D_1 & 0 \\
                0 & 0
            \end{pmatrix},
        R_1-I_s\right).
    \end{align*}
    The chart function $\varphi$ is a coordinate slice of the chart function $\widetilde{\varphi}$. Hence, $\N_{Z_\#}$ is an embedded submanifold of $\M$.
\end{proof}

\end{document}